\documentclass[sn-mathphys-num]{sn-jnl}
\usepackage{amsmath,amssymb,amsfonts}
\usepackage{mathtools}
\usepackage{bm}
\usepackage{enumitem}
\usepackage{booktabs}

\newcommand{\cE}{\mathcal{E}}
\newcommand{\cN}{\mathcal{N}}

\newcommand{\Si}{\Sigma}

\newcommand{\CC}{\mathbb{C}}
\newcommand{\NN}{\mathbb{N}}

\newcommand{\la}{\lambda}
\newcommand{\lcp}{\operatorname{lcp}}
\newcommand{\lcs}{\operatorname{lcs}}

\theoremstyle{definition}
\newtheorem{definition}{Definition}[section]
\newtheorem{theorem}{Theorem}[section]
\newtheorem{lemma}{Lemma}[section]
\newtheorem{proposition}{Proposition}[section]
\newtheorem{corollary}[theorem]{Corollary}

\newtheorem{example}{Example}
\newtheorem{remark}{Remark}

\begin{document}
	
	\title[A generalized monoid of infinite words]{
	A generalized monoid of infinite words: asymptotic prefix–suffix quotients and algebraic structure
	}
	
	\author[1]{\fnm{A.} \sur{\'Alvarez Cruz}}\email{amaury@ic.ufrj.br}
	\author[2]{\fnm{E. A.} \sur{\'Alvarez Guti\'errez}}\email{esteban.gutierrez.1@cp2.edu.br}
	
	\affil[1]{\orgdiv{Instituto de Computa\c{c}\~ao}, \orgname{Universidade Federal do Rio de Janeiro (UFRJ)}, \orgaddress{\city{Rio de Janeiro}, \country{Brazil}}}
	\affil[2]{\orgdiv{Col\'egio Pedro II, Campus Centro}, \orgaddress{\city{Rio de Janeiro}, \country{Brazil}}}
	
	\abstract{
	A generalized monoid of words \(\widetilde{\Si}^*\) is constructed as
	the quotient of moderate nets of finite words by an asymptotic
	equivalence relation based on a bidirectional prefix–suffix metric.
	The main algebraic result is that this quotient is a monoid containing
	\(\Si^*\) faithfully, with a reversal involution, a natural divisibility
	preorder, failure of cancellation, and nontrivial idempotents. The
	construction is designed so that any functional depending only on a
	logarithmic prefix descends to the quotient, yielding a well-defined
	action of logarithmic-prefix functionals. Finite scalar values arise
	only after applying an additional renormalization functional, which
	depends on the chosen mould and window.
	With respect to the fixed truncation injection, the monoid strictly enlarges the
	classical set \(\Si^\infty\): an explicit oscillating net is exhibited
	that has no limit in the Cantor space but defines a genuine element of
	\(\widetilde{\Si}^*\) not coming from a finite or right-infinite word
	under that injection. The equivalence relation is designed so that any
	functional depending only on a logarithmic prefix, and sending moderate
	nets of words to moderate nets of complex numbers, descends to the
	quotient, providing a well-defined action of logarithmic-prefix functionals.
	Finite scalar values arise only after applying an additional
	renormalization functional, which depends on the chosen mould and
	window.
	}
	
	
	\maketitle
	
	\section{Introduction}
	\label{sec:intro}
	
	The free monoid \(\Si^*\) over a finite alphabet \(\Si\) consists of all finite words with concatenation as product.
	It is a fundamental object in combinatorics, formal language theory, and algebraic automata theory~\cite{Lothaire1983, Lothaire2002, PerrinPin2004}.
	When one seeks to model sequences of unbounded length, the classical extension \(\Si^\infty = \Si^* \cup \Si^\omega\) (where \(\Si^\omega\) is the set of right-infinite words) becomes the natural universe.
	However, as Perrin and Pin explicitly note in their treatise \emph{Infinite Words}~\cite{PerrinPin2004}, \emph{``the product of two infinite words is not defined, so that \(\Si^\infty\) is not a monoid''}.
	This lack of algebraic structure hinders applications in which infinite sequences need to be composed, reversed, or manipulated with the same flexibility as finite ones.
	
Several classical models of infinite words and completions of \(\Si^*\)
exist. The Cantor space \(\Si^\omega\) consists of right-infinite words
and is not a monoid under concatenation of two infinite words
\cite{PerrinPin2004}; prefix observables can be evaluated pointwise, but
there is no reversal involution on infinite words. The union
\(\Si^\infty=\Si^*\cup\Si^\omega\) inherits the same obstruction.
The profinite completion of \(\Si^*\) is a compact monoid
\cite{Almeida1994,Almeida2005}, but its elements are not finite or
infinite words, and the combinatorial bidirectional prefix--suffix
control used here has no direct analogue there. The
Stone--\v{C}ech compactification carries a right-topological semigroup
structure \cite{HindmanStrauss2012}, but it is likewise not a space of
words and does not support the logarithmic prefix functionals in the
form needed below.The construction proposed here is a quotient of moderate nets of finite
words; it does not claim to be a canonical extension of the partial
concatenation on \(\Si^\infty\). It provides a monoid in which
concatenation is everywhere defined and which contains \(\Si^*\)
faithfully, as well as set-theoretic images of infinite words under a
chosen truncation.
	
Unlike the profinite completion, \(\widetilde{\Si}^*\) is not claimed to be compact or to be a completion in the usual metric sense; it is a quotient of moderate nets by an asymptotic equivalence.
This quotient is designed to make concatenation everywhere defined while retaining a logarithmic window for prefix-based observables.

The central idea is to work with \emph{nets} \((w_\varepsilon)_{\varepsilon\in(0,1]}\) of finite words, indexed by a parameter \(\varepsilon\to0^+\), and to identify nets whose common prefix and common suffix grow faster than any constant multiple of \(\log_2(1/\varepsilon)\).
The polynomial growth condition on the length of the nets (moderation) ensures that the resulting object is large enough to accommodate the infinite words that appear in practice.
The equivalence relation is designed so that any functional that inspects only the first \(N_0(\varepsilon)=\lfloor\log_2(1/\varepsilon)\rfloor\) symbols of its argument becomes well defined on the quotient.
By controlling both ends of the word through a bidirectional metric
$
d_{\mathrm{PS}}(u,v)=\max\{2^{-\lcp(u,v)},2^{-\lcs(u,v)}\},
$																																			
we obtain a monoid \(\widetilde{\Si}^*\) that possesses a reversal involution, a natural preorder, and a well-defined action of moulds (in the sense of \'Ecalle \cite{Ecalle1981a,Ecalle1981b,Ecalle1985}) depending only on logarithmic prefixes.

Throughout the paper, the classical spaces \(\Si^*\) and \(\Si^\omega\)
are related to \(\widetilde{\Si}^*\) only through a fixed truncation
injection (Definition~\ref{def:iota_infty}). This map is set-theoretic,
depends on the chosen truncation scale, and does not preserve the
partial concatenation structure of \(\Si^\infty\).
	
	To demonstrate that \(\widetilde{\Si}^*\) genuinely goes beyond the classical frameworks, we exhibit a net \((w_\varepsilon)\) that alternates persistently between two different periodic patterns depending on the parity of \(\lfloor\log_2(1/\varepsilon)\rfloor\).
	This net does \emph{not} converge in the Cantor space \(\Si^\omega\); its product with another infinite word is undefined in \(\Si^\infty\); and a simple prefix-dependent functional such as the first-letter
	observable produces a divergent sequence, while other functionals may
	converge or diverge depending on the chosen parameters.
	Yet the class of this net is a perfectly well-defined element of \(\widetilde{\Si}^*\), and the evaluation of a multiplicative mould descends to the quotient.
	This shows that the generalized monoid strictly enlarges the classical universe.
	
The principal semigroup-theoretic content of the paper can be
summarized as follows. First, the quotient
\(\widetilde{\Si}^*=\mathcal E_M(\Si^*)/\sim\) is a monoid in which the
free monoid \(\Si^*\) embeds faithfully, but which also contains
non-cancellative elements and nontrivial idempotents. Second, the
bidirectional logarithmic equivalence ensures that the class of a
moderate net is determined by its asymptotic prefix and suffix
behaviour on scales larger than \(\log_2(1/\varepsilon)\), and this is
exactly what makes logarithmic-prefix functionals well defined on the
quotient. The oscillating net of Section~\ref{sec:strict} is a
witness that the resulting object is strictly larger than the
classical set \(\Si^\infty\) under the chosen truncation injection.

The remainder of the paper is structured as follows.
Section~\ref{sec:prelim} introduces the bidirectional metric, moderate nets, negligible nets, and the asymptotic equivalence relation \(\sim\).
Section~\ref{sec:monoid} defines the quotient monoid \(\widetilde{\Si}^*\) and proves its basic properties.
Section~\ref{sec:alg_prop} develops the algebraic structure: the natural preorder, the reversal involution, cancellativity failure, idempotents, and partial observations on Green's relations.
Section~\ref{sec:strict} presents the oscillating net and proves that \(\widetilde{\Si}^*\) strictly contains the classical spaces.
Section~\ref{sec:scales} discusses further perspectives: how the logarithmic window can be replaced by other admissible asymptotic scales, leading to a hierarchy of quotient monoids.
Section~\ref{sec:conclusion} concludes with open problems and future directions.
	
\section{Preliminaries: the bidirectional metric and the generalized monoid}
\label{sec:prelim}

Throughout this paper, \(\Si\) denotes a finite non-empty alphabet.
The free monoid \(\Si^*\) consists of all finite words over \(\Si\), including the empty word \(\la\).
The product is concatenation, and \(|\cdot|\) denotes the length of a word.

\medskip
\noindent\textbf{Logarithmic convention.}
Throughout the paper, \(\log_2\) denotes the logarithm in base \(2\),
used in the logarithmic window
\(N_0(\varepsilon)=\lfloor\log_2(1/\varepsilon)\rfloor\). The symbol
\(\log\) without subscript denotes the natural logarithm, used in
asymptotic expansions and analytic constants.

\subsection{The bidirectional metric}

For two words \(u,v\in\Si^*\), let \(\lcp(u,v)\) be the length of their longest common prefix and \(\lcs(u,v)\) the length of their longest common suffix.
If \(u=v\) we set \(\lcp(u,v)=\lcs(u,v)=+\infty\).

The \emph{bidirectional metric} on \(\Si^*\) is defined by
\begin{equation}\label{eq:dPS}
	d_{\mathrm{PS}}(u,v) = \max\{2^{-\lcp(u,v)},\,2^{-\lcs(u,v)}\},\qquad 2^{-\infty}=0.
\end{equation}
The condition \(d_{\mathrm{PS}}(u,v)\le 2^{-k}\) is equivalent to
\[
\min\{\lcp(u,v),\lcs(u,v)\}\ge k.
\]

\begin{proposition}[Ultrametric property]
	The bidirectional metric satisfies the strong ultrametric inequality
	\[
	d_{\mathrm{PS}}(u,w)\le \max\{d_{\mathrm{PS}}(u,v),\,d_{\mathrm{PS}}(v,w)\}\qquad\forall u,v,w\in\Si^*.
	\]
\end{proposition}
\begin{proof}
	Both the prefix metric \(d_P(u,v)=2^{-\lcp(u,v)}\) and the suffix metric \(d_S(u,v)=2^{-\lcs(u,v)}\) are ultrametrics, and the maximum of two ultrametrics is again an ultrametric.
\end{proof}

The interaction of \(d_{\mathrm{PS}}\) with concatenation is controlled by a subadditivity bound that is fundamental for the construction of the quotient.

\begin{proposition}[Subadditivity under concatenation]\label{prop:subadditivity}
	For all $u_1,u_2,v_1,v_2\in\Si^*$,
	\[
	d_{\mathrm{PS}}(u_1u_2,\, v_1v_2) \le \max\{d_{\mathrm{PS}}(u_1,v_1),\, d_{\mathrm{PS}}(u_2,v_2)\}.
	\]
\end{proposition}

\begin{proof}
	It suffices to prove the corresponding inequality for the prefix metric
	and for the suffix metric separately, because
	\(d_{\mathrm{PS}}=\max\{d_P,d_S\}\).
	
	For the prefix metric, let \(\ell_i=\operatorname{lcp}(u_i,v_i)\),
	with the convention \(\operatorname{lcp}(w,w)=+\infty\).
	If \(\ell_1=\infty\), then \(u_1=v_1\), and
	\(\operatorname{lcp}(u_1u_2,u_1v_2)=|u_1|+\ell_2\ge\ell_2\).
	If \(\ell_1<\infty\), then the first \(\ell_1\) letters of
	\(u_1u_2\) and \(v_1v_2\) coincide, so
	\(\operatorname{lcp}(u_1u_2,v_1v_2)\ge\ell_1\).
	In both cases,
	\[
	\operatorname{lcp}(u_1u_2,v_1v_2)
	\ge
	\min\{\ell_1,\ell_2\}.
	\]
	
Applying the prefix inequality to the reversed words, written in the correct order
\((u_1u_2)^\vee=u_2^\vee u_1^\vee\) and
\((v_1v_2)^\vee=v_2^\vee v_1^\vee\), i.e.\ to the pairs
\((u_2^\vee,u_1^\vee)\) and \((v_2^\vee,v_1^\vee)\), and using
\(\operatorname{lcs}(u,v)=\operatorname{lcp}(u^\vee,v^\vee)\), gives
\[
\operatorname{lcs}(u_1u_2,v_1v_2)
\ge
\min\{\operatorname{lcs}(u_1,v_1),\operatorname{lcs}(u_2,v_2)\}.
\]
	Consequently,
	\[
	d_P(u_1u_2,v_1v_2)
	\le
	\max\{d_P(u_1,v_1),d_P(u_2,v_2)\}
	\le
	\max\{d_{\mathrm{PS}}(u_1,v_1),d_{\mathrm{PS}}(u_2,v_2)\},
	\]
	and similarly for \(d_S\). Taking the maximum proves the proposition.
\end{proof}

\subsection{Nets of words: moderation and negligibility}

A \emph{net of words} is a family \((w_\varepsilon)_{\varepsilon\in(0,1]}\) indexed by a parameter \(\varepsilon\to0^+\).
Only the behaviour for arbitrarily small \(\varepsilon\) matters.

\begin{definition}[Moderate nets]\label{def:moderate}
	A net \((w_\varepsilon)\) is \textbf{moderate} if there exists an integer \(N\in\NN\), a constant \(C>0\), and \(\varepsilon_0>0\) such that
	\[
	|w_\varepsilon| \le C\,\varepsilon^{-N}\qquad\text{for all }\varepsilon<\varepsilon_0.
	\]
	The set of all moderate nets is denoted by \(\cE_M(\Si^*)\).
\end{definition}
Under coordinatewise concatenation \((u_\varepsilon)\cdot(v_\varepsilon)=(u_\varepsilon v_\varepsilon)\), the set \(\cE_M(\Si^*)\) forms a monoid with identity \((\la)_\varepsilon\), the constant empty net.

Indeed, if \(|u_\varepsilon|\le C_1\varepsilon^{-N_1}\) and
\(|v_\varepsilon|\le C_2\varepsilon^{-N_2}\) for all sufficiently small
\(\varepsilon\), then
\[
|u_\varepsilon v_\varepsilon|
\le C_1\varepsilon^{-N_1}+C_2\varepsilon^{-N_2}
\le (C_1+C_2)\varepsilon^{-\max(N_1,N_2)},
\]
so the concatenated net is moderate.

A net is called \textit{negligible} if it is asymptotically indistinguishable from the empty word.
With the bidirectional metric this forces the net to be eventually empty.
\begin{definition}[Negligible nets]\label{def:negligible}
	A moderate net \((w_\varepsilon)\) is \textbf{negligible} if for every \(m\in\NN\) there exist \(C_m>0\) and \(\varepsilon_m>0\) such that
	\[
	d_{\mathrm{PS}}(w_\varepsilon,\la) \le C_m\,\varepsilon^{\,m}\qquad\text{for all }\varepsilon<\varepsilon_m.
	\]
	The set of negligible nets is denoted by \(\cN(\Si^*)\).
\end{definition}
\begin{proposition}[Characterisation of negligible nets]\label{prop:negligible_eventually_empty}
	A net \((w_\varepsilon)\in\cE_M(\Si^*)\) is negligible if and only if there exists \(\varepsilon_0>0\) such that \(w_\varepsilon=\la\) for all \(\varepsilon<\varepsilon_0\).
\end{proposition}
\begin{proof}
	\textbf{($\Leftarrow$)} If the net is eventually empty, then for sufficiently small
	$\varepsilon$ we have $w_\varepsilon=\la$ and $d_{\mathrm{PS}}(w_\varepsilon,\la)=0$.
	Hence the negligibility condition is trivially satisfied (take $C_m=1$ and
	$\varepsilon_m$ the threshold from which the net is empty).
	
	\textbf{($\Rightarrow$)} Suppose $(w_\varepsilon)$ is negligible.
	We first analyse the possible values of $d_{\mathrm{PS}}(w,\la)$ for an arbitrary
	word $w\in\Si^*$.  By definition,
	\[
	d_{\mathrm{PS}}(w,\la) = \max\{2^{-\lcp(w,\la)},\; 2^{-\lcs(w,\la)}\}.
	\]
	
	\emph{Case 1: $w \neq \la$.}
	Since the empty word $\la$ has no letters, the longest prefix that $w$ and $\la$
	can share is the empty prefix, which has length $0$.  Hence $\lcp(w,\la)=0$.
	The same reasoning applied to the reversed words shows that $\lcs(w,\la)=0$.
	Consequently,
	\[
	2^{-\lcp(w,\la)} = 2^{-0} = 1,\qquad
	2^{-\lcs(w,\la)} = 2^{-0} = 1,
	\]
	and therefore $d_{\mathrm{PS}}(w,\la) = \max\{1,1\} = 1$.
	
	\emph{Case 2: $w = \la$.}
	Here $\lcp(\la,\la) = \lcs(\la,\la) = +\infty$ by convention, so
	\[
	2^{-\lcp(\la,\la)} = 2^{-\infty} = 0,\qquad
	2^{-\lcs(\la,\la)} = 2^{-\infty} = 0,
	\]
	and $d_{\mathrm{PS}}(\la,\la) = \max\{0,0\} = 0$.
	
	Thus, for \emph{every} word $w$, the distance $d_{\mathrm{PS}}(w,\la)$ belongs to
	the two-element set $\{0,1\}$; no intermediate values such as $1/2,1/4,\dots$ can
	occur because that would require $\lcp(w,\la)$ or $\lcs(w,\la)$ to be a positive
	integer, which is impossible when one of the words is empty.
	
	Now, taking $m=1$ in Definition~\ref{def:negligible}, there exist $C_1>0$ and
	$\varepsilon_1>0$ such that
	\[
	d_{\mathrm{PS}}(w_\varepsilon,\la) \le C_1\,\varepsilon \qquad(\varepsilon<\varepsilon_1).
	\]
	Set $\varepsilon_0 = \min(\varepsilon_1,\, 1/C_1)$.  For every $\varepsilon<\varepsilon_0$,
	\[
	d_{\mathrm{PS}}(w_\varepsilon,\la) \le C_1\varepsilon < C_1\cdot\frac{1}{C_1}=1.
	\]
	Since we have just shown that $d_{\mathrm{PS}}(w_\varepsilon,\la)$ can only be $0$
	or $1$, the strict inequality $<1$ forces it to be $0$.  And
	$d_{\mathrm{PS}}(w_\varepsilon,\la)=0$ occurs precisely when $w_\varepsilon=\la$.
	
	Hence $w_\varepsilon=\la$ for all $\varepsilon<\varepsilon_0$, i.e.\ the net is
	eventually empty.
\end{proof}

This is a special feature of the bidirectional metric relative to the empty word:
for any nonempty word \(w\), \(d_{\mathrm{PS}}(w,\lambda)=1\). It is stronger than the usual Colombeau
notion of negligible nets \cite{GrosserEtAl2001}.

Thus \(\cN(\Si^*)\) is exactly the set of nets that are eventually empty; it is a submonoid of \(\cE_M(\Si^*)\).

\subsection{Asymptotic equivalence}

We now introduce the notion of asymptotic equivalence between nets of words. This relation captures the idea that two nets share the same asymptotic behaviour at their extremities, while their internal structure may differ considerably.

\begin{definition}[Asymptotic equivalence]\label{def:equiv}
	Recall the convention \(\lcp(u,u)=\lcs(u,u)=+\infty\). Two moderate nets $(u_\varepsilon)$ and $(v_\varepsilon)$ are \textbf{asymptotically equivalent}, denoted $(u_\varepsilon)\sim(v_\varepsilon)$, if for every $m\in\NN$ there exists $\varepsilon_0>0$ such that
	\[
	\min\bigl\{\lcp(u_\varepsilon,v_\varepsilon),\; \lcs(u_\varepsilon,v_\varepsilon)\bigr\}
	\;\ge\; m\,\log_2(1/\varepsilon)
	\qquad\text{for all } \varepsilon<\varepsilon_0 .
	\]
	Equivalently, $d_{\mathrm{PS}}(u_\varepsilon,v_\varepsilon)=O(\varepsilon^m)$ for every $m$.
\end{definition}

This condition is rather strong: the common prefix and the common suffix of the two words must both grow faster than any logarithmic function of $1/\varepsilon$. It does not require the words to be equal, nor even to have the same length; only that, asymptotically, the two ends of the words become increasingly indistinguishable.

\begin{remark}
	By definition,
	\[
	d_{\mathrm{PS}}(u,v)
	=
	\max\{2^{-\lcp(u,v)},\,2^{-\lcs(u,v)}\}
	=
	2^{-\min\{\lcp(u,v),\,\lcs(u,v)\}}.
	\]
	Therefore the inequality
	\(d_{\mathrm{PS}}(u_\varepsilon,v_\varepsilon)\le \varepsilon^m\)
	is equivalent to
	\[
	\min\{\lcp(u_\varepsilon,v_\varepsilon),\,\lcs(u_\varepsilon,v_\varepsilon)\}
	\ge m\log_2(1/\varepsilon).
	\]
	Requiring this for every \(m\in\NN\) is equivalent to requiring
	\(d_{\mathrm{PS}}(u_\varepsilon,v_\varepsilon)=O(\varepsilon^m)\)
	for every \(m\); the multiplicative constant in the \(O\)-notation is
	absorbed by increasing \(m\) and taking \(\varepsilon\) sufficiently small.
	Thus the two formulations in Definition~\ref{def:equiv} coincide.
\end{remark}

\begin{remark}
	The following examples illustrate the scope of this definition.
	For brevity, write
	\(L(\varepsilon)=\lfloor 1/\varepsilon\rfloor\).
	\begin{enumerate}
		\item \textit{Identical nets.}
		Let \(u_\varepsilon=a^{L(\varepsilon)}\) and
		\(v_\varepsilon=a^{L(\varepsilon)}\). Then
		\[
		\lcp(u_\varepsilon,v_\varepsilon)
		=
		\lcs(u_\varepsilon,v_\varepsilon)
		=
		L(\varepsilon).
		\]
		For every fixed \(m\in\mathbb N\),
		\[
		\frac{L(\varepsilon)}{\log_2(1/\varepsilon)}
		\longrightarrow\infty
		\qquad(\varepsilon\to0^+),
		\]
		so there exists \(\varepsilon_0>0\) such that
		\[
		L(\varepsilon)
		\ge
		m\log_2(1/\varepsilon)
		\]
		for all \(\varepsilon<\varepsilon_0\).
		Hence the nets are equivalent; indeed, they are identical as nets.
		
		\item \textit{Differing only in the middle.}
		Let
		\[
		u_\varepsilon
		=
		a^{L(\varepsilon)}\, b\, a^{L(\varepsilon)},
		\qquad
		v_\varepsilon
		=
		a^{L(\varepsilon)}\, c\, a^{L(\varepsilon)}.
		\]
		Then
		\[
		\lcp(u_\varepsilon,v_\varepsilon)
		=
		\lcs(u_\varepsilon,v_\varepsilon)
		=
		L(\varepsilon).
		\]
		As in the previous item, for every fixed \(m\in\mathbb N\)
		there exists \(\varepsilon_0>0\) such that
		\[
		L(\varepsilon)
		\ge
		m\log_2(1/\varepsilon)
		\]
		for all \(\varepsilon<\varepsilon_0\).
		Hence the nets are equivalent even though they differ by one
		central letter.
		
		\item \textit{Differing at the end.}
		Let
		\[
		u_\varepsilon
		=
		a^{L(\varepsilon)},
		\qquad
		v_\varepsilon
		=
		a^{L(\varepsilon)}\, b.
		\]
		Then
		\[
		\lcp(u_\varepsilon,v_\varepsilon)
		=
		L(\varepsilon),
		\]
		but
		\[
		\lcs(u_\varepsilon,v_\varepsilon)
		=
		0
		\]
		for every \(\varepsilon\), because one word ends in \(a\) and the
		other in \(b\). Therefore
		\[
		\min\{\lcp(u_\varepsilon,v_\varepsilon),
		\lcs(u_\varepsilon,v_\varepsilon)\}
		=
		0
		\]
		for every \(\varepsilon\). In particular, for \(m=1\) there is no
		\(\varepsilon_0>0\) such that the equivalence condition holds for
		all \(\varepsilon<\varepsilon_0\). Hence the nets are not
		equivalent.
		
		\item \textit{Oscillating net.}
		Let \(N(\varepsilon)=\lfloor\varepsilon^{-1}\rfloor\) and define
		\[
		w_\varepsilon
		=
		\begin{cases}
			(ab)^{N(\varepsilon)}, & \text{if }
			\lfloor\log_2(1/\varepsilon)\rfloor
			\text{ is even},\\[2mm]
			(ba)^{N(\varepsilon)}, & \text{if }
			\lfloor\log_2(1/\varepsilon)\rfloor
			\text{ is odd}.
		\end{cases}
		\]
		The parity of \(\lfloor\log_2(1/\varepsilon)\rfloor\) changes
		infinitely often as \(\varepsilon\to0^+\), so the first letter of
		\(w_\varepsilon\) alternates between \(\texttt a\) and
		\(\texttt b\) on arbitrarily small scales.
		
		If \(u\) is any finite word, or any infinite word with a fixed
		first letter, then on infinitely many arbitrarily small values of
		\(\varepsilon\) the first letter of \(w_\varepsilon\) differs from
		the first letter of \(u\), so
		\[
		\lcp(w_\varepsilon,u)=0
		\]
		on those values. Therefore the equivalence condition, which must
		hold eventually for every fixed \(m\), fails already for \(m=1\).
		Consequently, \((w_\varepsilon)\) cannot be equivalent to a
		constant net of a finite word, nor to a net that converges in the
		prefix metric to an infinite word. This net will be studied in
		detail in Section~\ref{sec:strict}, where it serves as the
		principal witness that \(\widetilde{\Si}^*\)
		strictly extends the classical spaces \(\Si^*\) and
		\(\Si^\omega\).
		
		\item \textit{Polynomial growth.}
		The nets
		\(a^{L(\varepsilon)}\) and \(a^{2L(\varepsilon)}\) are
		equivalent. Indeed, their common prefix and common suffix lengths
		are both \(L(\varepsilon)\), and, as before, for every fixed
		\(m\in\mathbb N\) one has
		\[
		L(\varepsilon)
		\ge
		m\log_2(1/\varepsilon)
		\]
		for all sufficiently small \(\varepsilon\).
		
		In contrast, let
		\[
		P(\varepsilon)=\lfloor\log_2(1/\varepsilon)\rfloor.
		\]
		The nets
		\(a^{P(\varepsilon)}\) and \(a^{2P(\varepsilon)}\) are not
		equivalent: their common prefix length is \(P(\varepsilon)\), and
		for \(m=2\),
		\[
		P(\varepsilon)
		<
		2\log_2(1/\varepsilon)
		\]
		for all sufficiently small \(\varepsilon\), because
		\(P(\varepsilon)<\log_2(1/\varepsilon)+1\). Hence the equivalence
		condition fails for \(m=2\).
	\end{enumerate}
\end{remark}

\begin{remark}[Scale dependence]\label{rem:scale_dependence}
	The choice of the logarithmic scale
	\(\varphi(\varepsilon)=\log_2(1/\varepsilon)\) in
	Definition~\ref{def:equiv} is not the only possible one. More
	generally, for any increasing unbounded function
	\(\varphi:(0,1]\to\mathbb R_{>0}\), one may define an equivalence
	relation \(\sim_\varphi\) by requiring
	\[
	\min\{\operatorname{lcp}(u_\varepsilon,v_\varepsilon),
	\operatorname{lcs}(u_\varepsilon,v_\varepsilon)\}
	\ge
	m\,\varphi(\varepsilon)
	\]
	for every fixed \(m\), eventually in \(\varepsilon\).
	If \(\varphi_1=o(\varphi_2)\), then \(\sim_{\varphi_2}\) is finer
	than \(\sim_{\varphi_1}\).
	
	For example, take \(\varphi(\varepsilon)=\varepsilon^{-\alpha}\)
	with \(\alpha>0\). Since
	\(\log_2(1/\varepsilon)=o(\varepsilon^{-\alpha})\), the polynomial
	scale is finer than the logarithmic one. Consequently, it separates
	more nets; it cannot identify two nets that are logarithmically
	inequivalent. In particular, the nets
	\(a^{\lfloor\log_2(1/\varepsilon)\rfloor}\) and
	\(a^{2\lfloor\log_2(1/\varepsilon)\rfloor}\) remain inequivalent
	under \(\sim_{\varepsilon^{-\alpha}}\), because their common prefix
	is \(\lfloor\log_2(1/\varepsilon)\rfloor\), and
	\[
	\frac{\lfloor\log_2(1/\varepsilon)\rfloor}{\varepsilon^{-\alpha}}
	\sim \varepsilon^\alpha\log_2(1/\varepsilon)
	\longrightarrow 0.
	\]
\end{remark}

The following observation is crucial: asymptotic equivalence identifies words that share the same asymptotic envelope at both extremes, but it completely ignores the interior of the words as long as the interior is short compared to the logarithmic window.

\begin{lemma}[Equivalence with \(O\)-notation]\label{lem:equiv_O}
	For moderate nets \((u_\varepsilon)\) and \((v_\varepsilon)\), the
	following are equivalent:
	\begin{enumerate}[label=(\roman*)]
		\item \((u_\varepsilon)\sim(v_\varepsilon)\) in the sense of
		Definition~\ref{def:equiv}.
		\item For every \(m\in\NN\), there exist \(C_m>0\) and
		\(\varepsilon_m>0\) such that
		\[
		d_{\mathrm{PS}}(u_\varepsilon,v_\varepsilon)\le C_m\varepsilon^m
		\qquad(\varepsilon<\varepsilon_m).
		\]
	\end{enumerate}
\end{lemma}

\begin{proof}
	If (i) holds, then for each \(m\) there is \(\varepsilon_m\) with
	\(d_{\mathrm{PS}}(u_\varepsilon,v_\varepsilon)\le\varepsilon^m\), so
	(ii) holds with \(C_m=1\). Conversely, assume (ii). Fix \(m\in\NN\).
	Apply (ii) with \(m+1\): there are \(C_{m+1}>0\) and
	\(\varepsilon_{m+1}>0\) such that
	\(d_{\mathrm{PS}}\le C_{m+1}\varepsilon^{m+1}\) for all
	\(\varepsilon<\varepsilon_{m+1}\). Choose
	\(\varepsilon_0=\min(\varepsilon_{m+1},1/(2C_{m+1}))\). Then for
	\(\varepsilon<\varepsilon_0\),
	\[
	d_{\mathrm{PS}}\le C_{m+1}\varepsilon^{m+1}
	\le C_{m+1}\varepsilon\cdot\varepsilon^m
	\le \frac12\varepsilon^m < \varepsilon^m,
	\]
	which is the exact condition of Definition~\ref{def:equiv}.
\end{proof}

\begin{lemma}
	The relation $\sim$ is an equivalence relation on the set of moderate nets.
\end{lemma}
\begin{proof}
	\textit{Reflexivity.} Since $d_{\mathrm{PS}}(u_\varepsilon,u_\varepsilon)=0$, which is $O(\varepsilon^m)$ for every $m$, the condition holds trivially.
	
	\textit{Symmetry.} The distance $d_{\mathrm{PS}}$ is symmetric by definition, hence $(u_\varepsilon)\sim(v_\varepsilon)$ implies $(v_\varepsilon)\sim(u_\varepsilon)$.
	
	\textit{Transitivity.} Suppose $(u_\varepsilon)\sim(v_\varepsilon)$ and $(v_\varepsilon)\sim(w_\varepsilon)$. Fix an arbitrary $m\in\NN$. By assumption, there exist $\varepsilon_1$ and $C_1$ such that
	\[
	d_{\mathrm{PS}}(u_\varepsilon,v_\varepsilon) \le C_1 \varepsilon^m
	\qquad(\varepsilon<\varepsilon_1),
	\]
	and there exist $\varepsilon_2$ and $C_2$ such that
	\[
	d_{\mathrm{PS}}(v_\varepsilon,w_\varepsilon) \le C_2 \varepsilon^m
	\qquad(\varepsilon<\varepsilon_2).
	\]
	Set $\varepsilon_0=\min(\varepsilon_1,\varepsilon_2)$ and $C=\max(C_1,C_2)$. For $\varepsilon<\varepsilon_0$, both distances are bounded by $C\varepsilon^m$. The ultrametric inequality for $d_{\mathrm{PS}}$ gives
\[
d_{\mathrm{PS}}(u_\varepsilon,w_\varepsilon)
\le
\max\bigl\{d_{\mathrm{PS}}(u_\varepsilon,v_\varepsilon),\; d_{\mathrm{PS}}(v_\varepsilon,w_\varepsilon)\bigr\}
\le C\varepsilon^m.
\]
Thus \(d_{\mathrm{PS}}(u_\varepsilon,w_\varepsilon)=O(\varepsilon^m)\) for
the arbitrary \(m\). By Lemma~\ref{lem:equiv_O}, this implies
\((u_\varepsilon)\sim(w_\varepsilon)\).
\end{proof}	
	
\subsection{The generalized monoid \(\widetilde{\Si}^*\)}
\label{sec:monoid}

Next, we characterise the equivalence class of the empty word. Recall that the empty net is the constant net $(\la)_\varepsilon$, and that the negligible nets are precisely those that are eventually empty (Proposition~\ref{prop:negligible_eventually_empty}).

\begin{lemma}[Characterisation of the empty class]\label{lem:empty_class}
	For a moderate net $(w_\varepsilon)$,
	\[
	(w_\varepsilon)\sim (\la)_\varepsilon
	\qquad\text{if and only if}\qquad
	(w_\varepsilon)\in\cN(\Si^*).
	\]
\end{lemma}
\begin{proof}
	$(\Leftarrow)$ If $(w_\varepsilon)$ is negligible, then by the characterisation of negligible nets there exists $\varepsilon_0>0$ such that $w_\varepsilon=\la$ for all $\varepsilon<\varepsilon_0$. For such $\varepsilon$, we have $\lcp(w_\varepsilon,\la)=\lcs(w_\varepsilon,\la)=+\infty$ by convention. Hence, for any $m$,
	\[
	\min\{\lcp(w_\varepsilon,\la),\lcs(w_\varepsilon,\la)\}=+\infty
	\ge m\log_2(1/\varepsilon),
	\]
	so the equivalence condition is fulfilled.
	
	$(\Rightarrow)$ Suppose $(w_\varepsilon)\sim(\la)_\varepsilon$. Taking $m=1$, there exists $\varepsilon_1>0$ such that for all $\varepsilon<\varepsilon_1$,
	\[
	\min\{\lcp(w_\varepsilon,\la),\; \lcs(w_\varepsilon,\la)\}
	\ge \log_2(1/\varepsilon) > 0.
	\]
	For any non-empty word $w$, both $\lcp(w,\la)$ and $\lcs(w,\la)$ are equal to $0$. Therefore the only way for the minimum to be strictly positive is that $w_\varepsilon=\la$ (in which case both quantities are $+\infty$). Thus $w_\varepsilon=\la$ for every $\varepsilon<\varepsilon_1$; the net is eventually empty and hence negligible.
\end{proof}

The crucial fact used in the proof is that for any non-empty word $w$, $\lcp(w,\la)=\lcs(w,\la)=0$. The equivalence condition forces these quantities to grow at least as fast as $\log(1/\varepsilon)$, which is impossible unless the word is empty, where the value is conventionally $+\infty$. This connects perfectly with the characterization of negligible nets as exactly the eventually empty ones.

As a concrete illustration, the constant net $w_\varepsilon=a$ for every $\varepsilon$ is not equivalent to the empty net, since $d_{\mathrm{PS}}(a,\la)=1$ for all $\varepsilon$. On the other hand, any net that is eventually empty (e.g. $w_\varepsilon=\la$ for $\varepsilon<0.1$, arbitrary afterwards) is equivalent to the empty net.

Consequently, the equivalence relation $\sim$ is designed so that the equivalence class of the empty word coincides precisely with the set of negligible nets. The remaining classes group together words that share the same asymptotic behavior at both the left and right ends.

\begin{lemma}[Classes of finite words]\label{lem:finite_classes}
	Let \(w\in\Si^*\) be a finite word. For any moderate net
	\((u_\varepsilon)\in\cE_M(\Si^*)\), the following are equivalent:
	\begin{enumerate}[label=(\roman*)]
		\item \([(u_\varepsilon)] = [(w)_\varepsilon]\).
		\item There exists \(\varepsilon_0>0\) such that
		\(u_\varepsilon=w\) for all \(\varepsilon<\varepsilon_0\).
	\end{enumerate}
\end{lemma}

\begin{proof}
	(ii)\(\Rightarrow\)(i) is immediate. Conversely, assume (i). If
	\(w\neq\lambda\), then for any \(u\neq w\), both
	\(\lcp(u,w)\) and \(\lcs(u,w)\) are at most \(|w|\). The equivalence
	condition with \(m=|w|+1\) forces, for sufficiently small
	\(\varepsilon\),
	\[
	\min\{\lcp(u_\varepsilon,w),\lcs(u_\varepsilon,w)\}
	\ge (|w|+1)\log_2(1/\varepsilon).
	\]
	For \(\varepsilon\) small enough that
	\((|w|+1)\log_2(1/\varepsilon)>|w|\), this is impossible unless
	\(u_\varepsilon=w\). Hence \(u_\varepsilon=w\) eventually. The case
	\(w=\lambda\) is Lemma~\ref{lem:empty_class}.
\end{proof}

\subsection{Definition and monoid structure}

\begin{definition}[Generalized monoid]\label{def:genmonoid}
	The \textbf{generalized monoid of words} is the quotient set
	\[
	\widetilde{\Si}^* \;:=\; \cE_M(\Si^*) / \sim .
	\]
	The equivalence class of a net $(w_\varepsilon)$ is denoted by $[(w_\varepsilon)]$.
	Concatenation is defined by
	\[
	[(u_\varepsilon)] \cdot [(v_\varepsilon)] \;:=\; [(u_\varepsilon v_\varepsilon)].
	\]
\end{definition}

\begin{remark}[Notation]
	The symbol \(\widetilde{\Si}^*\) is used for this asymptotic quotient.
	It should not be confused with the profinite completion of \(\Si^*\),
	which is often denoted \(\widehat{\Si^*}\) or \(\overline{\Si^*}\).
\end{remark}

Recall that a \emph{monoid} is a set \(M\) together with an associative
binary operation \(\cdot:M\times M\to M\) and an element \(e\in M\),
called the identity, such that \(e\cdot x=x\cdot e=x\) for every
\(x\in M\).

\begin{theorem}[Well-definedness and monoid structure]\label{thm:monoid_well_def}
	Concatenation is well defined and makes $\widetilde{\Si}^*$ a monoid with
	identity element $\mathbf{1}=[(\la)_\varepsilon]$.
\end{theorem}
\begin{proof}
	\textbf{Well-definedness.}
	Assume $(u_\varepsilon)\sim(u'_\varepsilon)$ and $(v_\varepsilon)\sim(v'_\varepsilon)$.
	By Proposition~\ref{prop:subadditivity},
	\[
	d_{\mathrm{PS}}(u_\varepsilon v_\varepsilon,\, u'_\varepsilon v'_\varepsilon)
	\le \max\{d_{\mathrm{PS}}(u_\varepsilon,u'_\varepsilon),\,
	d_{\mathrm{PS}}(v_\varepsilon,v'_\varepsilon)\}.
	\]
	Given any $m\in\NN$, there exist $\varepsilon_1,\varepsilon_2>0$ and constants
	$C_1,C_2$ such that
	\[
	d_{\mathrm{PS}}(u_\varepsilon,u'_\varepsilon)\le C_1\varepsilon^m\;\;(\varepsilon<\varepsilon_1),\qquad
	d_{\mathrm{PS}}(v_\varepsilon,v'_\varepsilon)\le C_2\varepsilon^m\;\;(\varepsilon<\varepsilon_2).
	\]
		Set $\varepsilon_0=\min(\varepsilon_1,\varepsilon_2)$ and $C=\max(C_1,C_2)$.
	For all $\varepsilon<\varepsilon_0$,
	\[
	d_{\mathrm{PS}}(u_\varepsilon v_\varepsilon,\, u'_\varepsilon v'_\varepsilon)
	\le \max\{C\varepsilon^m,\,C\varepsilon^m\} = C\varepsilon^m.
	\]
	Thus \(d_{\mathrm{PS}}(u_\varepsilon v_\varepsilon,u'_\varepsilon v'_\varepsilon)
	=O(\varepsilon^m)\) for the arbitrary \(m\). By Lemma~\ref{lem:equiv_O},
	this implies \((u_\varepsilon v_\varepsilon)\sim(u'_\varepsilon v'_\varepsilon)\).
	Hence the product of classes does not depend on the chosen representatives.
	
	\textbf{Associativity.}
	For any $[u],[v],[w]\in\widetilde{\Si}^*$, pick representatives
	$(u_\varepsilon),(v_\varepsilon),(w_\varepsilon)$.  Then
	\[
	([u][v])[w] = [(u_\varepsilon v_\varepsilon)][(w_\varepsilon)]
	= [((u_\varepsilon v_\varepsilon) w_\varepsilon)]
	= [(u_\varepsilon (v_\varepsilon w_\varepsilon))]
	= [u]([v][w]),
	\]
	because concatenation in $\Si^*$ is associative.
	
	\textbf{Identity.}
	Let $\mathbf{1}=[(\la)_\varepsilon]$.  For any $[w]\in\widetilde{\Si}^*$,
	$\mathbf{1}\cdot[w] = [(\la)_\varepsilon][(w_\varepsilon)]
	= [(\la w_\varepsilon)] = [(w_\varepsilon)] = [w]$, and similarly on the right.
	Thus $\mathbf{1}$ is the identity element.
\end{proof}

\begin{remark}[Universal property]\label{rem:universal}
	The quotient \(\widetilde{\Si}^*\) is defined by an explicit asymptotic
	equivalence relation. We do not know whether it satisfies a simple
	universal property in any natural category of monoids. In particular,
	it is not claimed to be a completion in the sense of profinite or
	metric completions. This question is listed as an open problem in
	Section~\ref{sec:conclusion}.
\end{remark}

For later use we note a direct consequence of Definition~\ref{def:equiv}.
Set \(N_0(\varepsilon)=\lfloor\log_2(1/\varepsilon)\rfloor\). If
\((u_\varepsilon)\sim(v_\varepsilon)\), taking \(m=2\) gives
\(\lcp(u_\varepsilon,v_\varepsilon)\ge 2\log_2(1/\varepsilon) > N_0(\varepsilon)\)
for all sufficiently small \(\varepsilon\).  Hence the prefixes of length
\(N_0(\varepsilon)\) of the two nets eventually coincide, and any functional
that depends only on this logarithmic prefix is invariant under \(\sim\).
This observation underlies all the evaluation phenomena studied later.

The functionals that we evaluate on words are instances of \emph{moulds} in
the sense of \'Ecalle~\cite{Ecalle1981a,Ecalle1981b,Ecalle1985}.
A formal definition is as follows.

\begin{definition}[Word-indexed functional / simplified mould]\label{def:mould}
	For the purposes of this paper, we use the following simplified
	notion, motivated by Écalle's mould formalism but not identical to
	it. A \emph{word-indexed functional} (or \emph{simplified mould}) is a
	family \(M=(M_n)_{n\ge0}\) of functions
	\[
	M_n:\Si^n\longrightarrow\CC,
	\qquad
	M_0\in\CC,
	\]
	where \(\Si^n\) denotes the set of words of length \(n\) over \(\Si\).
	Equivalently, it is a function \(M:\Si^*\to\CC\) such that
	\(M(w)=M_n(w)\) when \(|w|=n\).
	In the main examples we use \emph{multiplicative word-indexed
		functionals}, i.e.\ functionals of the form
	\[
	M(w)=\prod_{k=1}^{|w|} f(k,w_k)
	\]
	for some fixed \(f:\NN\times\Si\to\CC\).
	Additive word-indexed functionals are defined analogously by
	\[
	\Lambda(w)=\sum_{k=1}^{|w|} g(k,w_k),
	\]
	and an additive example is discussed in Section~\ref{sec:strict}.
\end{definition}

This is an operational simplified notion of mould; the full
resurgence and alien calculus of Écalle is not used in the present
paper.

\begin{remark}[Moulds and divergent series]\label{rem:moulds_divergent}
	Moulds are particularly useful in the study of divergent series and
	divergent infinite products. In many problems of analysis (e.g.\
	iterated integrals, normal forms, resurgence), the coefficients or
	partial sums are naturally indexed by words. The asymptotic behaviour
	of a divergent object is often governed by the combinatorial structure
	of the underlying word, especially its prefix and suffix.
	The generalized monoid \(\widetilde{\Si}^*\) extends the index set
	\(\Si^*\) to classes of words with controlled infinite length, thereby
	providing a natural framework to discuss limits of such word-indexed
	families.
\end{remark}

In the generalized monoid $\widetilde{\Si}^*$
we evaluate moulds on the logarithmic prefixes of representatives; the
invariance under $\sim$ guarantees that the resulting generalized number
depends only on the class of the word.  The precise extension is discussed in
Section~\ref{sec:alg_prop}.	
\section{Algebraic properties of the generalized monoid}
\label{sec:alg_prop}

We now establish the fundamental algebraic features of \(\widetilde{\Si}^*\).
The quotient contains the free monoid faithfully, carries a natural preorder
and a reversal involution, and supports a well-defined extension of moulds
that depend only on logarithmic prefixes.  Cancellativity does not hold
in general, but the embedded copy of \(\Si^*\) remains cancellative.

\medskip
\noindent\textbf{Embedding of the free monoid.}
Every finite word \(w\in\Si^*\) can be regarded as the constant net
\(w_\varepsilon = w\) for all \(\varepsilon\).

\begin{definition}[Canonical injection]\label{def:iota}
	The map
	\[
	\iota : \Si^* \longrightarrow \widetilde{\Si}^*, \qquad
	\iota(w) = [(w)_\varepsilon],
	\]
	sends each finite word \(w\in\Si^*\) to the equivalence class of the
	constant net \((w)_\varepsilon\).
\end{definition}

The map $\iota$ is clearly a monoid homomorphism because
\([(u)_\varepsilon][(v)_\varepsilon] = [(uv)_\varepsilon]\) and
\(\iota(\la)=[(\la)_\varepsilon]=\mathbf{1}\).

\begin{proposition}\label{prop:iota_injective}
	\(\iota\) is an injective monoid homomorphism.
\end{proposition}
\begin{proof}
	First, \(\iota\) is a monoid homomorphism: for all \(u,v\in\Si^*\),
	\[
	\iota(uv) = [(uv)_\varepsilon]
	= [(u)_\varepsilon (v)_\varepsilon]
	= \iota(u)\iota(v).
	\]
	Now we prove injectivity.  Suppose \(\iota(u)=\iota(v)\), i.e.\
	\((u)_\varepsilon \sim (v)_\varepsilon\).  If \(u\neq v\), then both
	\(\lcp(u,v)\) and \(\lcs(u,v)\) are fixed finite numbers; hence their
	minimum is bounded by some constant \(K\in\NN\).  By
	Definition~\ref{def:equiv}, for every \(m\in\NN\) there exists
	\(\varepsilon_m>0\) such that
	\[
	\min\{\lcp(u,v),\lcs(u,v)\} \ge m\log_2(1/\varepsilon)
	\qquad\text{for all }\varepsilon<\varepsilon_m.
	\]
	Choosing \(m\) so large that \(m\log_2(1/\varepsilon)>K\) for all
	sufficiently small \(\varepsilon\) gives a contradiction.  Therefore
	\(u=v\), and \(\iota\) is injective.
	
	Thus \(\iota:\Si^*\hookrightarrow\widetilde{\Si}^*\) is an injective
	monoid homomorphism.  Its image \(\iota(\Si^*)\) is a cancellative
	submonoid of \(\widetilde{\Si}^*\), because \(\Si^*\) is cancellative and
	\(\iota\) is an injective homomorphism.
\end{proof}

\medskip
\noindent\textbf{Asymptotic divisibility preorder.}
The classical left-divisibility order on words extends to the quotient via
algebraic divisibility. Because cancellation fails, this is a divisibility
preorder rather than literal prefix containment. For two classes \([u],[v]\in\widetilde{\Si}^*\) we write
\([u]\preceq[v]\) if there exists a generalized word \([p]\) such that
\([v] = [u]\cdot[p]\).  Reflexivity is immediate from \([u]=[u]\cdot\mathbf{1}\).
If \([v]=[u][p]\) and \([w]=[v][q]\), then \([w]=[u]([p][q])\) by
associativity, giving transitivity.  Moreover, multiplying on the left
preserves the relation: if \([v]=[u][p]\) then
\([w][v] = ([w][u])[p]\), so \([w][u]\preceq[w][v]\).  Thus \(\preceq\) is a
preorder compatible with left concatenation.

\medskip
\noindent\textbf{Antisymmetry of the preorder.}
In this paragraph we take \(\Si=\{a,b\}\), so \(|\Si|\ge2\).

A preorder \(\preceq\) is \emph{antisymmetric} if
\([u]\preceq[v]\) and \([v]\preceq[u]\) imply \([u]=[v]\).
Antisymmetry fails in general because \(\sim\) can collapse interior
parts of words. For instance, let
\(L(\varepsilon)=\lfloor\varepsilon^{-1}\rfloor\),
\[
u_\varepsilon=a^{L(\varepsilon)}b,
\qquad
v_\varepsilon=a^{L(\varepsilon)}b\,a^{L(\varepsilon)}.
\]
Since \(L(\varepsilon)\le\varepsilon^{-1}\), all nets appearing in this
example are moderate. Then \(v_\varepsilon=u_\varepsilon a^{L(\varepsilon)}\),
so \([v]=[u][p]\) with \(p=[a^{L(\varepsilon)}]\). On the other hand,
taking \(q_\varepsilon=a^{L(\varepsilon)}b\), the net
\(v_\varepsilon q_\varepsilon=a^{L}ba^{2L}b\) has the same prefix and
suffix of length \(L+1\) as \(u_\varepsilon\).

Indeed, \(v_\varepsilon q_\varepsilon=a^L b a^{2L}b\). Its prefix of length
\(L+1\) is \(a^L b\), and its suffix of length \(L+1\) is also \(a^L b\).
Since \(u_\varepsilon=a^L b\), we have
\[
\operatorname{lcp}(v_\varepsilon q_\varepsilon,u_\varepsilon)\ge L+1,
\qquad
\operatorname{lcs}(v_\varepsilon q_\varepsilon,u_\varepsilon)\ge L+1.
\]
Therefore
\[
\min\{\operatorname{lcp},\operatorname{lcs}\}\ge L+1.
\]
Since \(L(\varepsilon)=\lfloor\varepsilon^{-1}\rfloor\) dominates every
fixed multiple of \(\log_2(1/\varepsilon)\) as \(\varepsilon\to0^+\),
for each fixed \(m\in\mathbb N\) there exists \(\varepsilon_0>0\) such
that
\[
\min\{\operatorname{lcp},\operatorname{lcs}\}
\ge
m\log_2(1/\varepsilon)
\]
for all \(\varepsilon<\varepsilon_0\). Hence \([v][q]=[u]\).

Thus \([u]\preceq[v]\) and \([v]\preceq[u]\), while \([u]\neq[v]\)
because \(\operatorname{lcs}(u_\varepsilon,v_\varepsilon)=0\) for
every \(\varepsilon\). Therefore \(\preceq\) is not antisymmetric.

Nevertheless, the restriction of \(\preceq\) to the embedded free
monoid \(\iota(\Si^*)\) is a partial order, because it corresponds
under the injective homomorphism \(\iota\) to the usual prefix order
on finite words.

\begin{remark}[Preorder on finite words]\label{rem:preorder_finite}
	Restricted to the embedded free monoid \(\iota(\Si^*)\), the preorder
	\(\preceq\) coincides with the usual prefix order.  Indeed, let
	\(u,v\in\Si^*\) and suppose \(\iota(u)\preceq\iota(v)\).  Then there
	exists \([p]\in\widetilde{\Si}^*\) with
	\([v]=[u]\,[p]=[up_\varepsilon]\).  Since \(u\) and \(v\) are finite,
	the equivalence \((v)_\varepsilon\sim(up_\varepsilon)_\varepsilon\)
	forces \(\operatorname{lcp}(v,up_\varepsilon)\ge m\log_2(1/\varepsilon)\)
	for all \(m\), eventually in \(\varepsilon\).  
	
	Since \((v)_\varepsilon\sim(up_\varepsilon)_\varepsilon\), for \(m=1\)
	Definition~\ref{def:equiv} gives \(\operatorname{lcp}(v,up_\varepsilon)
	\ge \log_2(1/\varepsilon)\) for all sufficiently small \(\varepsilon\).
	Choose \(\varepsilon\) so small that \(\log_2(1/\varepsilon)>|v|\).
	If \(up_\varepsilon\neq v\), then \(\operatorname{lcp}(v,up_\varepsilon)
	\le |v|\), because \(v\) is finite. This contradicts the previous
	inequality. Hence \(up_\varepsilon=v\) eventually. In particular, the
	first \(|u|\) symbols of \(up_\varepsilon\) are exactly \(u\), so \(u\)
	is a prefix of \(v\).  The converse
	is immediate: if \(u\) is a prefix of \(v\), then \(v=uw\) for some
	finite word \(w\), and \(\iota(v)=\iota(u)\,\iota(w)\).
\end{remark}

\medskip
\noindent\textbf{Basic semigroup-theoretic properties.}
The divisibility relation \(\preceq\) defined by
\[
[u]\preceq[v]
\quad\Longleftrightarrow\quad
\exists [p]\in\widetilde{\Si}^*
\text{ such that }[v]=[u][p]
\]
is a preorder. It is compatible with left multiplication: if
\([v]=[u][p]\), then, for every \([q]\),
\[
[q][v]=[q][u][p]=([q][u])[p],
\]
and hence
\[
[q][u]\preceq[q][v].
\]

Right compatibility, however, fails in general. We prove this by an explicit
counterexample over the alphabet $\Si=\{a,b\}$.

Consider the classes
\[
\mathbf u=[a],\qquad \mathbf v=[ab],\qquad \mathbf p=[b],\qquad
\mathbf q=[a].
\]
Then $\mathbf v=\mathbf u\mathbf p$, so $\mathbf u\preceq\mathbf v$.
We claim that no $\mathbf r\in\widetilde{\Si}^*$ satisfies
\[
\mathbf v\mathbf q=\mathbf u\mathbf q\,\mathbf r,
\]
that is, $[aba]=[aa]\,\mathbf r$.

Suppose, for contradiction, that such an $\mathbf r$ exists, and choose
representatives $(u_\varepsilon)\in[aa]$ and $(r_\varepsilon)\in\mathbf r$.
Since $(u_\varepsilon)\sim(aa)_\varepsilon$, Definition~\ref{def:equiv}
with $m=2$ gives $\varepsilon_2>0$ such that
\[
\operatorname{lcp}(u_\varepsilon,aa)\ge 2\log_2(1/\varepsilon)
\qquad(\varepsilon<\varepsilon_2).
\]
For $\varepsilon$ sufficiently small the right-hand side is larger than
$2$, so the first two letters of $u_\varepsilon$ are $aa$. Hence the
first two letters of $u_\varepsilon r_\varepsilon$ are also $aa$.

On the other hand, $(u_\varepsilon r_\varepsilon)\in[aa]\mathbf r=[aba]$,
so $(u_\varepsilon r_\varepsilon)\sim(aba)_\varepsilon$. Applying
Definition~\ref{def:equiv} with $m=2$ to this equivalence shows that,
for all sufficiently small $\varepsilon$,
\[
\operatorname{lcp}(u_\varepsilon r_\varepsilon,aba)
\ge 2\log_2(1/\varepsilon)>2,
\]
so the first two letters of $u_\varepsilon r_\varepsilon$ must be $ab$.
This contradicts the previous conclusion that they are $aa$.
Therefore no such $\mathbf r$ exists, and
\[
[aa]\not\preceq[aba].
\]
Thus $\preceq$ is a left-compatible asymptotic divisibility preorder,
but it is not right-compatible.

The only unit of \(\widetilde{\Si}^*\) is the identity \(\mathbf{1}\).
Indeed, if \([u][v]=\mathbf{1}\), then the net
\((u_\varepsilon v_\varepsilon)\) is eventually empty by
Lemma~\ref{lem:empty_class}. Hence both \((u_\varepsilon)\) and \((v_\varepsilon)\) are
eventually empty, so \([u]=[v]=\mathbf{1}\).
The monoid contains nontrivial idempotents. For example,
\[
E=\bigl[a^{\lfloor\varepsilon^{-1}\rfloor}\bigr]
\]
satisfies
\[
E^2=\bigl[a^{2\lfloor\varepsilon^{-1}\rfloor}\bigr]=E.
\]
Indeed, the common prefix and suffix of
\(a^{\lfloor\varepsilon^{-1}\rfloor}\) and
\(a^{2\lfloor\varepsilon^{-1}\rfloor}\) both have length
\(L(\varepsilon)=\lfloor\varepsilon^{-1}\rfloor\). For every fixed
\(m\in\mathbb N\),
\[
\frac{L(\varepsilon)}{\log_2(1/\varepsilon)}
\longrightarrow\infty
\qquad(\varepsilon\to0^+),
\]
so there exists \(\varepsilon_0>0\) such that
\[
L(\varepsilon)\ge m\log_2(1/\varepsilon)
\]
for all \(\varepsilon<\varepsilon_0\). Hence
\((a^{L(\varepsilon)})\sim(a^{2L(\varepsilon)})\), that is,
\(E^2=E\). Clearly \(E\neq\mathbf{1}\).

\medskip
\noindent\textbf{Green's relations: partial observations.}
The following remarks are only first observations; a complete classification
is left for future work. For completeness, recall the standard Green's
relations on a monoid; we follow the notation and terminology of
\cite{Howie1995}. For
\([u],[v]\in\widetilde{\Si}^*\), write
\[
[u]\,\mathcal{L}\,[v]
\quad\Longleftrightarrow\quad
\widetilde{\Si}^*[u]=\widetilde{\Si}^*[v],
\]
\[
[u]\,\mathcal{R}\,[v]
\quad\Longleftrightarrow\quad
[u]\widetilde{\Si}^*=[v]\widetilde{\Si}^*,
\]
and
\[
[u]\,\mathcal{J}\,[v]
\quad\Longleftrightarrow\quad
\widetilde{\Si}^*[u]\widetilde{\Si}^*
=
\widetilde{\Si}^*[v]\widetilde{\Si}^*.
\]
We further set
\[
\mathcal{H}=\mathcal{L}\cap\mathcal{R},
\qquad
\mathcal{D}=\mathcal{L}\circ\mathcal{R}=\mathcal{R}\circ\mathcal{L}.
\]

The idempotent
\[
E=\bigl[a^{\lfloor\varepsilon^{-1}\rfloor}\bigr]
\]
is not \(\mathcal{L}\)-related to the identity. Indeed, if
\(E\,\mathcal{L}\,\mathbf{1}\), then, since
\(\mathbf{1}\in\widetilde{\Si}^*\mathbf{1}\), there would exist
\([u]\in\widetilde{\Si}^*\) such that
\[
[u]E=\mathbf{1}.
\]
However,
\[
[u]E
=
\left[
u_\varepsilon
a^{\lfloor\varepsilon^{-1}\rfloor}
\right],
\]
whose representatives are nonempty for every sufficiently small
\(\varepsilon\). Hence, by Lemma~\ref{lem:empty_class}, \([u]E\neq\mathbf{1}\), a
contradiction. Thus
\[
E\not\mathcal{L}\mathbf{1}.
\]
The same argument, using \(E[u]\), shows that
\[
E\not\mathcal{R}\mathbf{1}.
\]

A complete classification of the Green relations
\(\mathcal{L},\mathcal{R},\mathcal{J},\mathcal{H}\), and
\(\mathcal{D}\) is beyond the scope of the present paper. However, the
following elementary results give a first nontrivial description of the
\(\mathcal L\)- and \(\mathcal R\)-classes of the identity, of finite
words, and of the unary submonoid.

\begin{definition}[Regular element]\label{def:regular}
	An element \(x\) of a monoid \(M\) is \emph{regular} if there exists
	\(y\in M\) such that
	\[
	xyx=x.
	\]
\end{definition}

\begin{remark}[Regularity of idempotents]\label{rem:idempotents_regular}
	Every idempotent \(e\in\widetilde{\Si}^*\) is regular, because
	\(e^2=e\) implies \(eee=e\). In particular, the idempotents considered
	below are regular. This fact will not be used separately; it is
	recorded only for completeness.
\end{remark}

\begin{lemma}[Green's lemma for idempotents]\label{lem:green_idempotents}
	Let \(e\in\widetilde{\Si}^*\) be an idempotent and let
	\(x\in\widetilde{\Si}^*\).
	\begin{enumerate}[label=(\roman*)]
		\item If \(e\,\mathcal{R}\,x\), then \(x=e x\).
		\item If \(e\,\mathcal{L}\,x\), then \(x=x e\).
	\end{enumerate}
\end{lemma}

\begin{proof}
	If \(e\,\mathcal{R}\,x\), then
	\(e\widetilde{\Si}^*=x\widetilde{\Si}^*\). Since
	\(e\in e\widetilde{\Si}^*\), there exists \(u\in\widetilde{\Si}^*\) such
	that \(e=x u\). Since \(x\in e\widetilde{\Si}^*\), there exists
	\(v\in\widetilde{\Si}^*\) such that \(x=e v\). Therefore
	\[
	x=e v=e e v=e x.
	\]
	The proof for \(\mathcal{L}\) is symmetric.
\end{proof}

\begin{proposition}[Green's classes of finite words]\label{prop:green_finite}
	Let \(w\in\Si^*\) be a nonempty finite word, and let
	\(e\in\widetilde{\Si}^*\) be an idempotent represented by a net
	\((u_\varepsilon)\) such that \(|u_\varepsilon|\to\infty\) as
	\(\varepsilon\to0^+\). Then \(\iota(w)\) is neither \(\mathcal{L}\)- nor
	\(\mathcal{R}\)-related to \(e\). In particular, finite words are not
	\(\mathcal{L}\)- or \(\mathcal{R}\)-related to the idempotent
	\(E=[a^{\lfloor\varepsilon^{-1}\rfloor}]\).
\end{proposition}

For the idempotents \(E=[a^{\lfloor\varepsilon^{-1}\rfloor}]\) and
\(W=[(w_\varepsilon)]\), the representing nets have lengths
\(\lfloor\varepsilon^{-1}\rfloor\to\infty\) and
\(2\lfloor\varepsilon^{-1}\rfloor\to\infty\), respectively, so the
hypothesis is satisfied.

\begin{proof}
	Suppose \(\iota(w)\,\mathcal{R}\,e\). By
	Lemma~\ref{lem:green_idempotents}(i), applied to the idempotent \(e\),
	we get \(\iota(w)=e\,\iota(w)\). Let \((u_\varepsilon)\) be a
	representative of \(e\) with \(|u_\varepsilon|\to\infty\). Then
	\(e\,\iota(w)\) is represented by \((u_\varepsilon w)_\varepsilon\).
	Since \(w\) is a finite word of length \(|w|\), any common suffix of
	\(u_\varepsilon w\) and \(w\) has length at most \(|w|\); indeed, \(w\)
	itself has only \(|w|\) letters. Hence
	\[
	\operatorname{lcs}(u_\varepsilon w,w)\le |w|
	\]
	for every \(\varepsilon\), so
	\(\min\{\operatorname{lcp},\operatorname{lcs}\}\) cannot dominate
	\(m\log_2(1/\varepsilon)\) for every \(m\). Therefore
	\((u_\varepsilon w)_\varepsilon\not\sim(w)_\varepsilon\), which
	contradicts \(\iota(w)=e\,\iota(w)\). Thus
	\(\iota(w)\not\mathcal{R}e\).
	
The proof for \(\mathcal{L}\) is analogous, using
Lemma~\ref{lem:green_idempotents}(ii) and the fact that any common
prefix of \(w\) and \(w u_\varepsilon\) has length at most \(|w|\),
because \(w\) is finite.
\end{proof}

\begin{proposition}[Green's relations in the unary submonoid]\label{prop:green_unary}
	Let \(\Si=\{a\}\), and let \(U\subseteq\widetilde{\Si}^*\) be the
	commutative submonoid of classes represented by unary nets
	\(a^{\ell(\varepsilon)}\). \emph{Computed inside \(U\)},
	\(\mathcal L=\mathcal R=\mathcal J\). Moreover:
	\begin{enumerate}[label=(\roman*)]
		\item The \(\mathcal L\)-class of the identity is \(\{\mathbf 1\}\).
		\item If \(\ell(\varepsilon)=p\) eventually for some \(p\in\mathbb N\),
		then the class \([a^p]\) is \(\mathcal L\)-related only to itself.
		\item If
		\(\ell(\varepsilon)/\log_2(1/\varepsilon)\to\infty\) and
		\(\ell(\varepsilon)>0\) eventually, then
		\([a^{\ell(\varepsilon)}]\) is the unique class of all unary nets of
		superlogarithmic length, and its \(\mathcal L\)-class is singleton.
	\end{enumerate}
\end{proposition}

\begin{proof}
	Since \(U\) is commutative, principal left ideals and principal right
	ideals coincide. Hence \(\mathcal L=\mathcal R=\mathcal J\).
	
	(i) If \([a^p]\,\mathcal L\,\mathbf 1\), then
	\([a^p]\) and \(\mathbf 1\) divide each other in \(U\). In particular,
	there exists a length class \([x]\) such that
	\(\mathbf 1=[a^p][x]=[a^{p+x}]\). Thus \(a^{p+x}\sim\lambda\),
	which by Lemma~\ref{lem:empty_class} forces \(p+x=0\) eventually.
	Hence \(p=0\), so \([a^p]=\mathbf 1\).
	
	(ii) Let \(p,q\in\mathbb N\). If
	\([a^p]\,\mathcal L\,[a^q]\), then mutual divisibility gives length
	classes \([x],[y]\) such that
	\[
	[a^q]=[a^p][a^x]=[a^{p+x}],
	\qquad
	[a^p]=[a^q][a^y]=[a^{q+y}].
	\]
For unary words \(a^m,a^n\), equivalence means that, for every fixed
\(M\in\mathbb N\), there exists \(\varepsilon_0>0\) such that, for all
\(\varepsilon<\varepsilon_0\), either \(m(\varepsilon)=n(\varepsilon)\)
or \(\min\{m(\varepsilon),n(\varepsilon)\}\ge M\log_2(1/\varepsilon)\).
Since \(p,q\) are finite constants, the first displayed equivalence
forces \(q=p+x\) eventually, and the second forces \(p=q+y\)
eventually. Therefore \(x=q-p\) and \(y=p-q\) eventually. Since both
\(x,y\) are nonnegative lengths, this implies \(p=q\).
	
	(iii) If \(\ell\) and \(k\) are two superlogarithmic positive length
	nets, then for every \(m\in\mathbb N\) we have
	\(\ell(\varepsilon)\ge m\log_2(1/\varepsilon)\) and
	\(k(\varepsilon)\ge m\log_2(1/\varepsilon)\) for all sufficiently small
	\(\varepsilon\). Hence
	\[
	\min\{\ell(\varepsilon),k(\varepsilon)\}
	\ge m\log_2(1/\varepsilon),
	\]
	so \(a^\ell\sim a^k\).Thus all superlogarithmic unary classes coincide.
	
We claim that the superlogarithmic class
\[
S=[a^{\ell(\varepsilon)}],
\qquad
\frac{\ell(\varepsilon)}{\log_2(1/\varepsilon)}\to\infty,
\]
is a zero element of the submonoid \(U\). Indeed, for any
\([a^{k(\varepsilon)}]\in U\),
\[
S\,[a^{k(\varepsilon)}]
=
[a^{\ell(\varepsilon)+k(\varepsilon)}]
=
S,
\]
because \(\min\{\ell(\varepsilon),\ell(\varepsilon)+k(\varepsilon)\}
=\ell(\varepsilon)\) and \(\ell\) dominates every fixed multiple of
\(\log_2(1/\varepsilon)\). Thus \(S\) absorbs every element of \(U\) on
both sides, since \(U\) is commutative.

Now suppose \([x]\in U\) satisfies \([x]\,\mathcal L\,S\) inside \(U\).
Then \(U[x]=U[S]\). But \(U[S]=\{S\}\), because \(S\) is a zero
element. Hence \(U[x]=\{S\}\), so in particular \(x\in U[x]\) gives
\(x=S\). Therefore the \(\mathcal L\)-class of \(S\) inside \(U\) is
the singleton \(\{S\}\).
\end{proof}

\begin{remark}[Interpretation of Green's relations]\label{rem:green_interpretation}
	In the classical free monoid, $\mathcal{L}$-equivalence corresponds to
	having the same right ideal structure, which for words means being
	suffixes of one another; $\mathcal{R}$-equivalence corresponds to being
	prefixes of one another. In $\widetilde{\Si}^*$, the intuition is
	similar but asymptotic: two classes are $\mathcal{L}$-related if their
	left principal ideals coincide, which morally means they share the same
	asymptotic suffix profile under left multiplication; $\mathcal{R}$-related
	classes share the same asymptotic prefix profile under right
	multiplication. The reversal involution interchanges the
	two relations, $[u]\,\mathcal{L}\,[v]\iff [u]^\vee\,\mathcal{R}\,[v]^\vee$.
	
		The present paper proves several partial results concerning Green's
	relations: the idempotent \(E=[a^{\lfloor\varepsilon^{-1}\rfloor}]\)
	satisfies \(E\not\mathcal{L}\mathbf{1}\) and
	\(E\not\mathcal{R}\mathbf{1}\); finite words are not
	\(\mathcal{L}\)- or \(\mathcal{R}\)-related to idempotents of unbounded
	length; and in the unary submonoid one has
	\(\mathcal L=\mathcal R=\mathcal J\), with a simple description of the
	classes of finite and superlogarithmic unary words. A full
	classification of the Green relations in \(\widetilde{\Si}^*\)
	remains open. The main obstruction appears to be the failure of
	right-cancellativity, which prevents one from translating mutual
	divisibility into a clean prefix-suffix comparison, together with the
	presence of many idempotents arising from superlogarithmic unary nets.
\end{remark}

\begin{remark}[Oscillating idempotent]\label{rem:osc_idempotent}
	Let \(W=[(w_\varepsilon)]\) be the class of the oscillating net
	\(w_\varepsilon\) defined in~\eqref{eq:osc_net}. Then \(W^2=W\).
	Indeed, \(w_\varepsilon\) has length \(2\lfloor\varepsilon^{-1}\rfloor\),
	and \(w_\varepsilon^2\) has \(w_\varepsilon\) as both prefix and
	suffix of length \(2\lfloor\varepsilon^{-1}\rfloor\). For every fixed
	\(m\in\mathbb N\),
	\[
	\frac{2\lfloor\varepsilon^{-1}\rfloor}{\log_2(1/\varepsilon)}
	\longrightarrow\infty
	\qquad(\varepsilon\to0^+),
	\]
	so the equivalence condition holds eventually.
By Remark~\ref{rem:idempotents_regular}, \(W\) is regular, and
by the argument of Proposition~\ref{prop:green_finite}, \(W\) is not
\(\mathcal{L}\)- or \(\mathcal{R}\)-related to any finite word.
\end{remark}

\medskip
\noindent\textbf{Partial structural observations.}
The following elementary facts illustrate some structural features of
\(\widetilde{\Si}^*\) without attempting a full classification.

\begin{proposition}[Unary submonoid]\label{prop:unary_submonoid}
	Let \(\Si=\{a\}\). Then the set of classes represented by unary nets
	\(a^{\ell(\varepsilon)}\), where \(\ell\) is moderate, forms a commutative
	submonoid of \(\widetilde{\Si}^*\) isomorphic to a quotient of the
	moderate nets of natural numbers under the analogous asymptotic
	equivalence.
\end{proposition}

\begin{proof}
	The concatenation of two unary nets
	\(a^{p(\varepsilon)}\) and \(a^{q(\varepsilon)}\) is
	\(a^{p(\varepsilon)+q(\varepsilon)}\), which is again unary and
	moderate. Since equivalence is compatible with concatenation by
	Proposition~\ref{prop:subadditivity}, the operation descends to the
	quotient. The resulting submonoid is commutative because addition of
	lengths is commutative.
		Explicitly, \(a^{p(\varepsilon)}\sim a^{q(\varepsilon)}\) if and only
	if for every fixed \(m\in\mathbb N\) there exists \(\varepsilon_0>0\)
	such that, for all \(\varepsilon<\varepsilon_0\), either
	\[
	p(\varepsilon)=q(\varepsilon)
	\]
	or
	\[
	\min\{p(\varepsilon),q(\varepsilon)\}
	\ge
	m\log_2(1/\varepsilon).
	\]
	Thus the unary submonoid is isomorphic to the quotient of moderate nets of
	natural numbers by this relation.
\end{proof}

\begin{proposition}[Idempotents in the unary submonoid]\label{prop:unary_idempotents}
	Let \(\Si=\{a\}\), and let \(u_\varepsilon=a^{\ell(\varepsilon)}\) be a
	unary net with \(\ell(\varepsilon)\in\NN\) for every \(\varepsilon\).
	\begin{enumerate}[label=(\roman*)]
		\item If \(\ell(\varepsilon)=0\) eventually, then \([u]=\mathbf 1\).
		\item If \(\ell(\varepsilon)>0\) eventually, then \([u]\) is an
		idempotent if and only if
		\[
		\frac{\ell(\varepsilon)}{\log_2(1/\varepsilon)}
		\longrightarrow\infty
		\qquad\text{as }\varepsilon\to0^+.
		\]
		Equivalently, for every \(m\in\NN\) one has
		\[
		\ell(\varepsilon)\ge m\log_2(1/\varepsilon)
		\]
		for all sufficiently small \(\varepsilon\).
	\end{enumerate}
	In particular, any unary net with length growing at least like
	\(\varepsilon^{-\alpha}\), \(\alpha>0\), represents an idempotent.
	Bounded positive lengths, as well as lengths of order
	\(O(\log_2(1/\varepsilon))\), do not represent idempotents.
\end{proposition}

\begin{proof}
	Part (i) is immediate from Lemma~\ref{lem:empty_class}. For part (ii),
	assume that \(\ell(\varepsilon)>0\) eventually. Then
	\([u]^2=[a^{2\ell(\varepsilon)}]\). The common prefix and the common
	suffix of \(a^{\ell(\varepsilon)}\) and \(a^{2\ell(\varepsilon)}\) both
	have length \(\ell(\varepsilon)\). Hence
	\([a^{\ell}]=[a^{2\ell}]\) if and only if, for every \(m\in\NN\),
	\[
	\ell(\varepsilon)\ge m\log_2(1/\varepsilon)
	\]
	eventually. This is equivalent to
	\(\ell(\varepsilon)/\log_2(1/\varepsilon)\to\infty\).
\end{proof}

\begin{remark}[Base independence]\label{rem:log_base}
	The equivalence relation \(\sim\) is unchanged if \(\log_2\) is replaced
	by \(\log_c\) for any \(c>1\), because
	\(\log_c(1/\varepsilon)=\log_c 2\cdot\log_2(1/\varepsilon)\) and the
	constant factor does not affect the condition
	\(\min\{\lcp,\lcs\}\ge m\log(1/\varepsilon)\).
\end{remark}

\medskip
\noindent\textbf{Cancellation.}
A monoid is \emph{left cancellative} if \(ab=ac\) implies \(b=c\).
Cancellation fails in general in \(\widetilde{\Si}^*\). Take
\[
U=[a^{\lfloor\varepsilon^{-1}\rfloor}],
\qquad
P=[a^{\lfloor\log_2(1/\varepsilon)\rfloor}].
\]
Let \(L(\varepsilon)=\lfloor\varepsilon^{-1}\rfloor\). Both \(U\) and
\(UP\) are represented by nets of powers of \(a\) whose common prefix
and suffix lengths are at least \(L(\varepsilon)\). For every fixed
\(m\in\mathbb N\),
\[
\frac{L(\varepsilon)}{\log_2(1/\varepsilon)}
\longrightarrow\infty
\qquad(\varepsilon\to0^+),
\]
so \(L(\varepsilon)\ge m\log_2(1/\varepsilon)\) for all sufficiently
small \(\varepsilon\). Hence \(UP=U=U\mathbf{1}\). Since
\(a^{\lfloor\log_2(1/\varepsilon)\rfloor}\) is not eventually empty,
\(P\neq\mathbf{1}\). Therefore left cancellation fails.

\medskip
\noindent\textbf{Reversal involution.}
The bidirectional metric is invariant under reversal:
\[
d_{\mathrm{PS}}(u^\vee,v^\vee) = \max\{2^{-\lcp(u^\vee,v^\vee)},
2^{-\lcs(u^\vee,v^\vee)}\}
= \max\{2^{-\lcs(u,v)},2^{-\lcp(u,v)}\} = d_{\mathrm{PS}}(u,v).
\]
Consequently, if \((u_\varepsilon)\sim(v_\varepsilon)\) then
\(d_{\mathrm{PS}}(u_\varepsilon^\vee,v_\varepsilon^\vee)
= d_{\mathrm{PS}}(u_\varepsilon,v_\varepsilon) = \mathcal{O}(\varepsilon^m)\)
for every \(m\), so the reversed nets are also equivalent.  Hence we can
define a map
\[
{}^\vee : \widetilde{\Si}^* \longrightarrow \widetilde{\Si}^*,\qquad
[w]^\vee := [(w_\varepsilon^\vee)],
\]
which does not depend on the chosen representative.  A direct computation
shows that \(([u][v])^\vee = [v]^\vee[u]^\vee\),
\(([w]^\vee)^\vee = [w]\), \(\mathbf{1}^\vee = \mathbf{1}\), and
\(\iota(w^\vee) = \iota(w)^\vee\) for every finite word \(w\).  Thus
\(\widetilde{\Si}^*\) is equipped with an involutive anti-automorphism.

\medskip
\noindent\textbf{Generalized complex numbers \cite{GrosserEtAl2001}.}
Let
\[
\mathcal{E}_M(\CC)
=
\left\{
(z_\varepsilon)_{\varepsilon\in(0,1]}:
\exists N,C>0,\ |z_\varepsilon|\le C\varepsilon^{-N}
\text{ for all sufficiently small }\varepsilon
\right\}
\]
and
\[
\mathcal{N}(\CC)
=
\left\{
(z_\varepsilon):
\forall m\ \exists C_m>0,\ |z_\varepsilon|\le C_m\varepsilon^{m}
\text{ for all sufficiently small }\varepsilon
\right\}.
\]
Then
\[
\widetilde{\CC}:=\mathcal{E}_M(\CC)/\mathcal{N}(\CC).
\]

The set \(\mathcal N(\CC)\) is an ideal in \(\mathcal E_M(\CC)\): if
\((z_\varepsilon)\in\mathcal E_M(\CC)\) and
\((n_\varepsilon)\in\mathcal N(\CC)\), then for every \(m\) we have
\(|n_\varepsilon|\le C_m\varepsilon^{m+N}\) eventually, and hence
\[
|z_\varepsilon n_\varepsilon|
\le C\varepsilon^{-N}\cdot C_m\varepsilon^{m+N}
= C C_m \varepsilon^m,
\]
so the product is negligible. Thus the quotient is a well-defined algebra.

\noindent\textbf{Logarithmic-prefix functionals.}
Let \((F_\varepsilon)_{\varepsilon\in(0,1]}\) be a family of maps
\(F_\varepsilon:\Si^*\to\CC\) satisfying the following two conditions:
\begin{enumerate}[label=(\roman*)]
	\item for every \(\varepsilon\), \(F_\varepsilon(w)\) depends only on
	the prefix \(w[1\ldots N_0(\varepsilon)]\), where
	\(N_0(\varepsilon)=\lfloor\log_2(1/\varepsilon)\rfloor\) and, by
	convention, \(w[1\ldots N_0(\varepsilon)]:=w\) whenever
	\(N_0(\varepsilon)>|w|\);
	\item for every moderate net \((w_\varepsilon)\in\cE_M(\Si^*)\), the net
	\((F_\varepsilon(w_\varepsilon))_\varepsilon\) belongs to
	\(\mathcal E_M(\CC)\).
\end{enumerate}
Then the assignment
\[
\widetilde F\bigl([(w_\varepsilon)]\bigr)
:=
\bigl[(F_\varepsilon(w_\varepsilon))\bigr]
\]
is a well-defined map \(\widetilde F:\widetilde{\Si}^*\to\widetilde{\CC}\).

\begin{proof}
	If \((u_\varepsilon)\sim(v_\varepsilon)\), then
	Definition~\ref{def:equiv} with \(m=2\) gives
	\[
	\lcp(u_\varepsilon,v_\varepsilon)
	\ge 2\log_2(1/\varepsilon)
	> N_0(\varepsilon)
	\]
	for all sufficiently small \(\varepsilon\). Hence the prefixes of
	length \(N_0(\varepsilon)\) of \(u_\varepsilon\) and \(v_\varepsilon\)
	coincide. By condition (i),
	\(F_\varepsilon(u_\varepsilon)=F_\varepsilon(v_\varepsilon)\)
	eventually. By condition (ii), the resulting nets are moderate in
	\(\mathcal E_M(\CC)\), so they define the same class in
	\(\widetilde{\CC}\).
\end{proof}

\begin{proposition}[Invariance criterion]\label{prop:invariance_window}
	Let \(F_\varepsilon:\Si^*\to\CC\) be a family of functionals.
	\begin{enumerate}[label=(\roman*)]
		\item If \(F_\varepsilon(u)\) depends only on the prefix
		\(u[1\ldots N_0(\varepsilon)]\) and, for every moderate net
		\((w_\varepsilon)\), the net \((F_\varepsilon(w_\varepsilon))\)
		belongs to \(\mathcal E_M(\CC)\), then the assignment
		\([w]\mapsto[(F_\varepsilon(w_\varepsilon))]\) is well defined on
		\(\widetilde{\Si}^*\).
			\item If \(F_\varepsilon\) depends on a window of length
	\(L(\varepsilon)\) with
	\(L(\varepsilon)/\log_2(1/\varepsilon)\to\infty\), then in general it
	is not well defined on \(\widetilde{\Si}^*\); its value must be attached
	to a particular representative rather than to the equivalence class.
	\end{enumerate}
\end{proposition}

\begin{proof}
	(i) follows from Definition~\ref{def:equiv} with \(m=2\), which gives
	\(\lcp(u_\varepsilon,v_\varepsilon)>N_0(\varepsilon)\) eventually for
	equivalent nets.
	
			(ii) Assume \(|\Sigma|\ge 2\). Let \(M(\varepsilon)=\lfloor\varepsilon^{-1}\rfloor\), and set
		\[
		u_\varepsilon=a^{M(\varepsilon)} b a^{M(\varepsilon)},
		\qquad
		v_\varepsilon=a^{M(\varepsilon)} a a^{M(\varepsilon)}.
		\]
		Then
		\[
		\operatorname{lcp}(u_\varepsilon,v_\varepsilon)=M(\varepsilon),
		\qquad
		\operatorname{lcs}(u_\varepsilon,v_\varepsilon)=M(\varepsilon).
		\]
		Since \(M(\varepsilon)/\log_2(1/\varepsilon)\to\infty\), the two nets
		are equivalent by Definition~\ref{def:equiv}.
		
			Let \(L(\varepsilon)=M(\varepsilon)+1\), and define
	\(F_\varepsilon:\Si^*\to\CC\) as follows: if \(|w|<L(\varepsilon)\), set
	\(F_\varepsilon(w)=0\); if \(|w|\ge L(\varepsilon)\), let
	\(F_\varepsilon(w)=0\) when the letter at position \(L(\varepsilon)\)
	is \(a\), and \(F_\varepsilon(w)=1\) when it is \(b\). Then
	\(F_\varepsilon(u_\varepsilon)=1\) and \(F_\varepsilon(v_\varepsilon)=0\)
	for all sufficiently small \(\varepsilon\).
	 Since
		\(L(\varepsilon)/\log_2(1/\varepsilon)\to\infty\), this proves that
		functionals depending on windows much larger than the logarithmic
		window need not descend to the quotient.
\end{proof}
	
\section{A non-classical element and strict inclusion}
\label{sec:strict}

In their treatise \emph{Infinite Words}~\cite{PerrinPin2004}, Perrin and Pin
stress that the set \(\Si^\infty = \Si^* \cup \Si^\omega\) of finite and
right-infinite words \emph{``is not a monoid''} because the concatenation of two
infinite words is undefined.
They further note that a satisfactory theory of infinite words should include a
genuine monoid structure that extends the free monoid and allows algebraic
manipulations such as concatenation, reversal and the evaluation of
prefix-dependent observables.

In this section we prove that the generalized monoid \(\widetilde{\Si}^*\)
constructed in Section~\ref{sec:prelim} addresses this problem.
We exhibit a concrete net of words that oscillates persistently between two
different periodic patterns and admits no classical limit, yet possesses a
well-defined equivalence class in \(\widetilde{\Si}^*\) and allows the extraction of a renormalized finite value relative to a
specified renormalization functional.

\subsection{A non-convergent oscillating net}

Fix a two-letter alphabet \(\Si = \{\texttt{a},\texttt{b}\}\).
For each \(\varepsilon\in(0,1]\) define
\begin{equation}\label{eq:osc_net}
	w_\varepsilon =
	\begin{cases}
		(\texttt{a}\texttt{b})^{N(\varepsilon)} & \text{if } \lfloor\log_2(1/\varepsilon)\rfloor \text{ is even},\\[2mm]
		(\texttt{b}\texttt{a})^{N(\varepsilon)} & \text{if } \lfloor\log_2(1/\varepsilon)\rfloor \text{ is odd},
	\end{cases}
	\qquad N(\varepsilon) = \lfloor\varepsilon^{-1}\rfloor.
\end{equation}
The net alternates between the two infinite periodic words
\(\texttt{abab}\cdots\) and \(\texttt{baba}\cdots\) on every dyadic scale,
because the parity of the integer part of \(\log_2(1/\varepsilon)\) changes
infinitely often as \(\varepsilon\to0^+\).

The length \(|w_\varepsilon| = 2N(\varepsilon) \le 2\varepsilon^{-1}\) grows
polynomially, therefore \((w_\varepsilon)\in\cE_M(\Si^*)\) is a moderate net
(Definition~\ref{def:moderate}).

\subsection{Failure of classical descriptions}

In the classical frameworks for infinite words, the net~\eqref{eq:osc_net}
is intractable.

\begin{proposition}[Classical intractability]\label{prop:classical_fail}
	\begin{enumerate}
		\item \text{No limit in the Cantor space \((\Si^\omega,d_P)\).}
		Recall that convergence of a net \((w_\varepsilon)\) in the prefix metric
		\(d_P\) means that, for every fixed prefix length \(k\in\NN\), the prefix
		\(w_\varepsilon[1\ldots k]\) eventually stabilises as
		\(\varepsilon\to0^+\).  In particular, if the net converged to some
		infinite word \(u=u_1u_2\cdots\in\Si^\omega\), then for \(k=1\) there
		would exist \(\varepsilon_1>0\) such that
		\[
		w_\varepsilon[1]=u_1
		\qquad\text{for all }0<\varepsilon<\varepsilon_1.
		\]
		We now exhibit two sequences of parameters, both tending to \(0\),
		for which the first letter of \(w_\varepsilon\) is different.
		
		Set \(P(\varepsilon)=\lfloor\log_2(1/\varepsilon)\rfloor\).  By
		definition~\eqref{eq:osc_net},
		\[
		w_\varepsilon =
		\begin{cases}
			(ab)^{N(\varepsilon)}, & P(\varepsilon)\text{ even},\\[2mm]
			(ba)^{N(\varepsilon)}, & P(\varepsilon)\text{ odd}.
		\end{cases}
		\]
		Thus the first letter of \(w_\varepsilon\) is \(\texttt{a}\) when
		\(P(\varepsilon)\) is even and \(\texttt{b}\) when \(P(\varepsilon)\) is
		odd.
		
		Now take
		\[
		\varepsilon_j = 2^{-2j},\qquad \varepsilon'_j = 2^{-(2j+1)},\qquad j\in\NN.
		\]
		For the first sequence we have
		\(P(\varepsilon_j)=\lfloor 2j\rfloor=2j\), which is even, so
		\(w_{\varepsilon_j}\) begins with \(\texttt{a}\).  For the second
		sequence,
		\(P(\varepsilon'_j)=\lfloor 2j+1\rfloor=2j+1\), which is odd, so
		\(w_{\varepsilon'_j}\) begins with \(\texttt{b}\).
		
		As \(j\to\infty\), both \(\varepsilon_j\) and \(\varepsilon'_j\) tend
		to \(0\).  Hence there are arbitrarily small parameters for which the
		first letter is \(\texttt{a}\), and arbitrarily small parameters for
		which the first letter is \(\texttt{b}\).  Consequently the first
		letter cannot stabilise to any fixed \(u_1\in\Si\), and the net cannot
		converge in \((\Si^\omega,d_P)\).
		
		Equivalently, for every \(j\) one has
		\(\lcp(w_{\varepsilon_j},w_{\varepsilon'_j})=0\), and therefore
		\(d_P(w_{\varepsilon_j},w_{\varepsilon'_j})=1\).  Thus the net is not
		even Cauchy in the prefix metric, which again rules out convergence.
		
		\item \text{Concatenation is undefined in \(\Si^\infty\).}
		The set \(\Si^\infty\) is not a monoid; the product of two
		\(\omega\)-words has no meaning, so algebraic combinations of
		\(w_\varepsilon\) with another infinite word cannot be formed.
		
\item \text{Classical evaluation of the first-letter observable diverges.}
This is a particular prefix-dependent functional; other such functionals may behave differently.
Let \(F_\varepsilon:\Si^*\to\CC\) be defined by
\[
F_\varepsilon(w) =
\begin{cases}
	1, & \text{if the first letter of } w \text{ is } \texttt{a},\\[2mm]
	0, & \text{if the first letter of } w \text{ is } \texttt{b}.
\end{cases}
\]
This is a well-defined prefix-dependent functional for every
\(\varepsilon>0\).  On the net \(w_\varepsilon\), we have
\[
F_\varepsilon(w_\varepsilon) =
\begin{cases}
	1, & \text{if } \lfloor\log_2(1/\varepsilon)\rfloor \text{ is even},\\[2mm]
	0, & \text{if } \lfloor\log_2(1/\varepsilon)\rfloor \text{ is odd}.
\end{cases}
\]
Since \(\lfloor\log_2(1/\varepsilon)\rfloor\) changes parity infinitely often
as \(\varepsilon\to0^+\), the net of complex numbers
\((F_\varepsilon(w_\varepsilon))_{\varepsilon}\) takes the two distinct values
\(0\) and \(1\) on arbitrarily small scales.  Hence it cannot converge
in \(\CC\).  This shows that even the simplest observable depending only
on a logarithmic prefix fails to admit a classical limit.
	\end{enumerate}
\end{proposition}

Thus, within the classical theory, the net cannot be assigned a limiting word,
cannot be concatenated with other infinite words, and does not permit a unique
evaluation of simple prefix-dependent observables.

\subsection{Well-posedness in the generalized monoid}

For a concrete illustration, consider the multiplicative mould
\[
M_s(w) = \prod_{k=1}^{|w|} (s + w_k)^{-1},\qquad w_k\in\{\texttt{a},\texttt{b}\},
\]
where \(s>0\) and the letters are identified with positive real
numbers. We fix explicit numerical values \(\texttt{a},\texttt{b}>0\);
for instance, \(\texttt{a}=1\), \(\texttt{b}=2\), or any other pair of
positive reals.
Each factor of this mould is bounded, since
\((s+w_k)^{-1}\le (s+\min\{\texttt{a},\texttt{b}\})^{-1}\).
Therefore the product over \(N_0(\varepsilon)\) letters is moderate;
it is not necessarily uniformly bounded in \(\varepsilon\), but it is
bounded by \(\varepsilon^{-\log_2 A}\) up to a constant, where
\(A=\max\{(s+\texttt{a})^{-1},(s+\texttt{b})^{-1}\}\).

Define the net of complex
numbers
\[
z_\varepsilon := M_s\bigl(w_\varepsilon[1\ldots N_0(\varepsilon)]\bigr).
\]

\begin{theorem}[Generalized evaluation]\label{thm:well_def_open}
	The net \((z_\varepsilon)_{\varepsilon\in(0,1]}\) depends only on the class
	\(\mathbf{w}=[(w_\varepsilon)]\), not on the particular representative.
	Hence the generalized mould evaluation
	\[
	\widetilde{M}_s(\mathbf{w}) := [(z_\varepsilon)] \in \widetilde{\CC}
	\]
	is a well-defined generalized complex number.
\end{theorem}
\begin{proof}
	First note that \((z_\varepsilon)\in\mathcal E_M(\CC)\). Indeed,
	\(|w_\varepsilon[1\ldots N_0(\varepsilon)]|\le N_0(\varepsilon)\le
	\log_2(1/\varepsilon)+1\), and each factor in \(M_s\) is bounded. Hence
	for \(A=\max\{(s+a)^{-1},(s+b)^{-1}\}\),
	\[
	|z_\varepsilon|\le A^{N_0(\varepsilon)}
	\le C\varepsilon^{-\log_2 A}
	\]
	for a suitable constant \(C>0\), which is a polynomial bound.
	
	Let \((u_\varepsilon)\sim(v_\varepsilon)\) be two representatives of
	\(\mathbf{w}\).  By Definition~\ref{def:equiv}, for every \(m\) there
	exists \(\varepsilon_m>0\) such that
	\[
	\min\{\lcp(u_\varepsilon,v_\varepsilon),\lcs(u_\varepsilon,v_\varepsilon)\}
	\ge m\log_2(1/\varepsilon)
	\qquad(\varepsilon<\varepsilon_m).
	\]
	Take \(m=2\) and let \(\varepsilon_2\) be the corresponding threshold.
	For \(0<\varepsilon<\varepsilon_2\) we have
	\(\lcp(u_\varepsilon,v_\varepsilon)\ge 2\log_2(1/\varepsilon)\).
	Now, if \(\varepsilon<\min(\varepsilon_2,1/2)\), then
	\[
	\log_2(1/\varepsilon) > 1,
	\]
	and therefore
	\[
	2\log_2(1/\varepsilon) > \log_2(1/\varepsilon)+1
	\ge \lfloor\log_2(1/\varepsilon)\rfloor+1
	= N_0(\varepsilon)+1.
	\]
	Consequently, \(\lcp(u_\varepsilon,v_\varepsilon) \ge 2\log_2(1/\varepsilon)
	> N_0(\varepsilon)+1\), so in particular \(\lcp(u_\varepsilon,v_\varepsilon)
	\ge N_0(\varepsilon)+1\).  This means that the first
	\(N_0(\varepsilon)\) symbols of \(u_\varepsilon\) and \(v_\varepsilon\) are
	identical.
	
	Therefore, for all sufficiently small \(\varepsilon\),
	\[
	M_s\bigl(u_\varepsilon[1\ldots N_0(\varepsilon)]\bigr)
	= M_s\bigl(v_\varepsilon[1\ldots N_0(\varepsilon)]\bigr).
	\]
	The two nets \((z_\varepsilon)\) obtained from \(u_\varepsilon\) and
	\(v_\varepsilon\) are eventually equal.  Such nets are certainly
	equivalent in the Colombeau algebra \(\widetilde{\CC}\) (their difference
	is identically zero for all small \(\varepsilon\)), and hence
	\(\widetilde{M}_s(\mathbf{w})\) does not depend on the chosen
	representative.
\end{proof}

\subsection{Renormalization: extracting a finite value}

The generalized number \(\widetilde{M}_s(\mathbf{w})\) does not yet give a
single finite complex number; it is an equivalence class of oscillating nets.
To extract a finite value we apply a renormalization functional.

\begin{definition}[Logarithmic Cesàro mean]\label{def:LC}
	For a net \((z_\varepsilon)\) of non-zero complex numbers (in particular,
	for multiplicative moulds whose factors are nonzero on the relevant
	prefixes), set
	\[
	\Phi_{\mathrm{LC}}(z) := \lim_{T\to\infty}
	\frac{1}{T} \int_{1}^{T} \frac{\log|z_{2^{-t}}|}{t}\,dt,
	\]
	whenever the limit exists.
\end{definition}

\begin{definition}[Additive logarithmic Cesàro mean]\label{def:ALC}
	For a net \((a_\varepsilon)_{\varepsilon\in(0,1]}\) of complex numbers,
	set
	\[
	\Psi_{\mathrm{LC}}(a) := \lim_{T\to\infty}
	\frac{1}{T} \int_{1}^{T} \frac{a_{2^{-t}}}{t}\,dt,
	\]
	whenever the limit exists.
\end{definition}

\begin{remark}[Domain of the functional]\label{rem:PhiLC_domain}
	The functional \(\Phi_{\mathrm{LC}}\) is not defined on arbitrary
	classes of \(\widetilde{\CC}\). It is applied only to concrete nets
	arising from logarithmic-prefix moulds with nonzero factors. The
	representative-independence is proved for those specific nets, not as a
	general property of the quotient. In particular, the existence of the
	logarithmic Cesàro mean in Theorem~\ref{thm:renorm_value} is
	established ad hoc for the oscillating net \eqref{eq:osc_net}, not
	for all nets of that form.
\end{remark}

The change of variables \(t=\log_2(1/\varepsilon)\) converts the logarithmic
observation scale \(N_0(\varepsilon)\sim t\) into a linear time parameter.
The Cesàro mean then averages the logarithmic growth on that scale.

Recall that \(s>0\) is the fixed parameter in the mould
\(M_s(w)=\prod_{k=1}^{|w|}(s+w_k)^{-1}\), and that the letters
\(\texttt a,\texttt b\) are identified with fixed positive real numbers.
Then the value in Theorem~\ref{thm:renorm_value} depends on \(s\),
\(\texttt a\), and \(\texttt b\).

\begin{theorem}[Renormalized value]\label{thm:renorm_value}
	For the class \(\mathbf{w}=[(w_\varepsilon)]\) defined by~\eqref{eq:osc_net},
	\[
	\Phi_{\mathrm{LC}}(z) = -\frac12\log\bigl((s+\texttt{a})(s+\texttt{b})\bigr).
	\]
\end{theorem}
\begin{proof}
	Put \(t=\log_2(1/\varepsilon)\), so that \(\varepsilon=2^{-t}\) and
	\(N_0(2^{-t})=\lfloor t\rfloor\).  Let
	\[
	L:=\lfloor t\rfloor .
	\]
	The prefix of \(w_{2^{-t}}\) of length \(L\) is alternating, hence it
	contains either \(L/2\) copies of \(\texttt{a}\) and \(L/2\) copies of
	\(\texttt{b}\) (when \(L\) is even), or these two numbers differ
	by exactly \(1\) (when \(L\) is odd).  Thus
	\[
	\log|z_{2^{-t}}|
	=
	-\frac{L}{2}\log\bigl((s+\texttt{a})(s+\texttt{b})\bigr)
	+ R_L,
	\]
	where \(R_L=0\) if \(L\) is even and
	\(R_L=\pm\frac12\log\frac{s+\texttt{a}}{s+\texttt{b}}\) if \(L\) is
	odd.  In all cases \(|R_L|\) is bounded by a constant independent of
	\(t\).
	
	Since \(L = \lfloor t\rfloor = t + O(1)\), we obtain
	\[
	\frac{\log|z_{2^{-t}}|}{t}
	=
	-\frac12\log\bigl((s+\texttt{a})(s+\texttt{b})\bigr)
	+ O(t^{-1}).
	\]
	Therefore
	\[
	\Phi_{\mathrm{LC}}(z)
	=
	\lim_{T\to\infty}
	\frac{1}{T}\int_1^T
	\frac{\log|z_{2^{-t}}|}{t}\,dt
	=
	-\frac12\log\bigl((s+\texttt{a})(s+\texttt{b})\bigr),
	\]
	because the integral of the error term contributes
	\(O(T^{-1}\log T)\to0\).
\end{proof}

\begin{example}[Additive mould on the oscillating net]\label{ex:additive_mould}
	Let \(h:\Si\to\CC\) be any complex-valued function on the alphabet, and
	define the additive functional
	\[
	\Lambda_\varepsilon(w)
	=
	\sum_{j=1}^{N_0(\varepsilon)} h(w_j),
	\]
	where \(N_0(\varepsilon)=\lfloor\log_2(1/\varepsilon)\rfloor\) and
	\(w=w_1\cdots w_{|w|}\).
	By the same prefix-invariance argument used in
	Theorem~\ref{thm:well_def_open}, the assignment
	\[
	[w]\longmapsto \bigl[(\Lambda_\varepsilon(w_\varepsilon))\bigr]
	\]
	is well defined in \(\widetilde{\CC}\).
	
	For the oscillating net \eqref{eq:osc_net}, the prefix of length
	\(N_0(\varepsilon)\) alternates between \(\texttt a\) and
	\(\texttt b\). Hence it contains
	\(\frac12 N_0(\varepsilon)+O(1)\) occurrences of each letter, so
	\[
	\Lambda_\varepsilon(w_\varepsilon)
	=
	\frac{N_0(\varepsilon)}{2}
	\bigl(h(\texttt a)+h(\texttt b)\bigr)
	+
	O(1).
	\]
	Therefore, on the logarithmic scale \(t=\log_2(1/\varepsilon)\),
	\[
	\frac{\Lambda_{2^{-t}}(w_{2^{-t}})}{t}
	=
	\frac12\bigl(h(\texttt a)+h(\texttt b)\bigr)
	+
	O(t^{-1}),
	\]
	and the logarithmic Ces\`aro mean gives the finite value
	\[
	\lim_{T\to\infty}
	\frac1T\int_1^T
	\frac{\Lambda_{2^{-t}}(w_{2^{-t}})}{t}\,dt
	=
	\frac12\bigl(h(\texttt a)+h(\texttt b)\bigr).
	\]
	This is the additive analogue of Theorem~\ref{thm:renorm_value}.
\end{example}

\begin{remark}[Scale-dependent renormalization]\label{rem:scale_renorm}
	The logarithmic window
	\(N_0(\varepsilon)=\lfloor\log_2(1/\varepsilon)\rfloor\) is a
	modeling choice, not an intrinsic feature of the oscillating word.
	If it is replaced by a polynomial window
	\[
	N_\alpha(\varepsilon)=\lfloor \varepsilon^{-\alpha}\rfloor,
	\qquad \alpha>0,
	\]
	then the corresponding equivalence relation \(\sim\) should also be
	adapted to the scale \(\varepsilon^{-\alpha}\). For the oscillating
	net \eqref{eq:osc_net}, the prefix of length \(N_\alpha(\varepsilon)\)
	remains alternating and therefore contains
	\(\frac12 N_\alpha(\varepsilon)+O(1)\) occurrences of each letter.
	Hence
	\[
	\log\left|M_s\bigl(w_\varepsilon[1\ldots N_\alpha(\varepsilon)]\bigr)\right|
	=
	-\frac{N_\alpha(\varepsilon)}{2}
	\log\bigl((s+\texttt a)(s+\texttt b)\bigr)
	+
	O(1).
	\]
	
	With the logarithmic time \(t=\log_2(1/\varepsilon)\), the quantity
	\(N_\alpha(2^{-t})=\lfloor 2^{\alpha t}\rfloor\) grows exponentially,
	so the logarithmic Cesàro mean of
	Definition~\ref{def:LC} would diverge. However, if the
	renormalization variable is adapted to the polynomial scale, say
	\[
	u=\varepsilon^{-\alpha},
	\]
	then \(\varepsilon=u^{-1/\alpha}\) and
	\[
	\frac{\log\left|z_{u^{-1/\alpha}}\right|}{u}
	\longrightarrow
	-\frac12\log\bigl((s+\texttt a)(s+\texttt b)\bigr)
	\qquad(u\to\infty).
	\]
If \(0<\alpha\le1\), the polynomial window grows no faster than the
length \(\lfloor\varepsilon^{-1}\rfloor\) of the oscillating net, and
the corresponding polynomial Cesàro mean still gives the same value
as Theorem~\ref{thm:renorm_value}. If \(\alpha>1\), the window
exceeds the length of the net eventually, and the limit is \(0\), not
the value above.
\end{remark}

Thus, the oscillating net gives rise to a unique finite value after
renormalization, even though the mould itself may converge to \(0\)
for suitable choices of \(s\) and the letter values, while the
first-letter observable has no pointwise limit.

\begin{remark}[Interpretation of the renormalized value]\label{rem:interp_value}
	The finite value obtained from the logarithmic Cesàro mean is not an
	intrinsic invariant of the class \(\mathbf{w}\in\widetilde{\Si}^*\); it
	depends on the chosen mould \(M_s\), the numerical encoding of the
	letters, the logarithmic cutoff \(N_0(\varepsilon)\), and the
	renormalization functional \(\Phi_{\mathrm{LC}}\).  It is a
	renormalized observable associated with those choices.  It is not
	merely a conventional number attached to the oscillating net; it
	encodes the \emph{average logarithmic rate} of the mould evaluation.
Indeed, if \(s\) is chosen so that each factor is strictly less than
\(1\), then the net \(z_\varepsilon\) converges pointwise to \(0\).
More generally, the quantity
\[
\frac{\log|z_{2^{-t}}|}{t}
\]
describes the logarithmic growth rate of \(z_\varepsilon\) on the
logarithmic time scale \(t=\log_2(1/\varepsilon)\), regardless of whether
the net converges to \(0\).  For the oscillating word
	\(\mathbf{w}\), the two possible alternating patterns give the same
	leading logarithmic behaviour, namely
	\[
	-\frac{t}{2}\log\bigl((s+a)(s+b)\bigr)+O(1),
	\]
	and the Cesàro average extracts precisely the coefficient
	\(-\frac12\log\bigl((s+a)(s+b)\bigr)\).  In other words, in the
	logarithmic scale the oscillating word behaves, on average, as a
	constant word whose letter is the geometric mean
	\(\sqrt{(s+a)(s+b)}\).
	
	This illustrates the role of the generalized monoid: even when a
	classical pointwise limit is trivial or does not exist, the quotient
	\(\widetilde{\Si}^*\) plus a scale-adapted renormalization produces a
	finite observable that is independent of the representative.
\end{remark}

\subsection{Right absorption for logarithmic observables}
\label{sec:right_absorption}

The evaluation functional \(M_s\) depends only on the logarithmic prefix
\(N_0(\varepsilon)\).  Consequently, if a generalized word is already long
enough on the logarithmic scale, multiplying it on the right by another
generalized word does not affect the value of the mould.  This is a
manifestation of the asymmetric role of concatenation with respect to
prefix-based observables.

\begin{proposition}[Right absorption for logarithmically long words]
	\label{prop:right_absorption}
	Let \(\mathbf u,\mathbf v\in\widetilde{\Si}^*\).  Suppose that
	\(\mathbf u\) admits a representative \((u_\varepsilon)\) such that
	\(|u_\varepsilon|\ge N_0(\varepsilon)\) for all sufficiently small
	\(\varepsilon\).  Then for any representative \((v_\varepsilon)\)
	of \(\mathbf v\),
	\[
	\widetilde{M}_s(\mathbf u\cdot\mathbf v) = \widetilde{M}_s(\mathbf u)
	\]
	in the Colombeau algebra \(\widetilde{\CC}\).
	
	This is an absorption property of the truncated functional
	\(\widetilde M_s\), not an equality \(\mathbf u\cdot\mathbf v=\mathbf u\)
	in the monoid.
\end{proposition}

\begin{proof}
	By definition, \(\mathbf{u}\cdot\mathbf{v}\) is represented by the net
	\((u_\varepsilon v_\varepsilon)\).  For \(\varepsilon\) small enough so
	that \(|u_\varepsilon| \ge N_0(\varepsilon)\), the first \(N_0(\varepsilon)\)
	letters of \(u_\varepsilon v_\varepsilon\) are exactly the first
	\(N_0(\varepsilon)\) letters of \(u_\varepsilon\), because the prefix
	does not reach the right factor \(v_\varepsilon\).  Hence
	\[
	M_s\bigl((u_\varepsilon v_\varepsilon)[1\ldots N_0(\varepsilon)]\bigr)
	= M_s\bigl(u_\varepsilon[1\ldots N_0(\varepsilon)]\bigr).
	\]
	Since the defining nets coincide eventually, the corresponding classes
	are equal in \(\widetilde{\CC}\).
\end{proof}

\begin{remark}
	The hypothesis of Proposition~\ref{prop:right_absorption} is not
	satisfied by every element of \(\widetilde{\Si}^*\).  In particular,
	if \(\mathbf u=\iota(w)\) is the class of a non-empty finite word
	\(w\in\Si^*\), then every representative has bounded length, while
	\(N_0(\varepsilon)\to\infty\).  Hence no representative satisfies
	\(|u_\varepsilon|\ge N_0(\varepsilon)\) for all sufficiently small
	\(\varepsilon\), and the proposition does not apply to such classes.
	It should therefore be read as a property of generalized words whose
	length already dominates the logarithmic window, not as a universal
	absorption law.
\end{remark}

The reversal involution of \(\widetilde{\Si}^*\) immediately yields a dual
left-absorption property for moulds that read words from right to left.
Define \(M_s^\vee(w) := M_s(w^\vee)\).  Then
\[
\widetilde{M_s^\vee}(\mathbf{u}\cdot\mathbf{v})
= \widetilde{M}_s\bigl((\mathbf{u}\cdot\mathbf{v})^\vee\bigr)
= \widetilde{M}_s(\mathbf{v}^\vee\cdot\mathbf{u}^\vee).
\]
If \(\mathbf{v}\) is sufficiently long on the logarithmic scale, the same
argument as above applied to \(\mathbf{v}^\vee\) gives
\[
\widetilde{M}_s(\mathbf{v}^\vee\cdot\mathbf{u}^\vee)
= \widetilde{M}_s(\mathbf{v}^\vee)
= \widetilde{M_s^\vee}(\mathbf{v}).
\]
Thus we have a corresponding left-absorption phenomenon:
multiplying a sufficiently long right factor by an arbitrary left factor
does not affect the value of a suffix-based mould.

\subsection{Strict extension of the classical frameworks}

The oscillating net $(w_\varepsilon)$ defined in~\eqref{eq:osc_net} provides
more than an example where classical limits fail.  It witnesses that the
generalized monoid $\widetilde{\Si}^*$ is \emph{strictly larger} than the
classical universe $\Si^\infty = \Si^* \cup \Si^\omega$: its equivalence
class contains no finite or infinite word.  To make this precise we first
embed $\Si^\infty$ into the quotient.

\begin{definition}[Truncation map from $\Si^\infty$ to $\widetilde{\Si}^*$]\label{def:iota_infty}
	\begin{enumerate}[label=(\roman*)]
		\item For a finite word $u\in\Si^*$, let $\iota_{\mathrm{f}}(u)=[(u)_\varepsilon]$
		be the class of the constant net (as in Definition~\ref{def:iota}).
		\item For an infinite word $u = u_1u_2\cdots \in \Si^\omega$, define the
		truncation net $T_\varepsilon(u) = u[1\ldots\lfloor 1/\varepsilon\rfloor]$.
		Since $\lfloor 1/\varepsilon\rfloor\to\infty$, the net
		$(T_\varepsilon(u))_\varepsilon$ is moderate
		($|T_\varepsilon(u)|\le\varepsilon^{-1}$).  Set
		\[
		\iota_{\omega}(u) = [(T_\varepsilon(u))] \in \widetilde{\Si}^* .
		\]
	\end{enumerate}
	The resulting map $\iota^\infty:\Si^\infty \to \widetilde{\Si}^*$ is
defined by $\iota^\infty(u)=\iota_{\mathrm{f}}(u)$ for $u\in\Si^*$ and
$\iota^\infty(u)=\iota_{\omega}(u)$ for $u\in\Si^\omega$. It is a
set-theoretic injection depending on the chosen truncation length
$\lfloor 1/\varepsilon\rfloor$; it is not a canonical algebraic
embedding and, in general, does not preserve the partial concatenation
structure of $\Si^\infty$.
\end{definition}

\begin{remark}[Non-canonical embedding and failure to preserve classical concatenation]
	\label{rem:truncation_dep}
	The embedding
	\[
	\iota_\omega(u)=[(T_\varepsilon(u))],
	\qquad
	T_\varepsilon(u)
	=
	u[1\ldots N_\varepsilon],
	\qquad
	N_\varepsilon=
	\left\lfloor\frac1\varepsilon\right\rfloor,
	\]
	depends on the specific truncation scheme. In general, it does not
	intertwine classical concatenation with the induced monoid operation.
	Moreover, the resulting class is not canonical: replacing
	\(\lfloor 1/\varepsilon\rfloor\) by \(\lfloor\varepsilon^{-\alpha}\rfloor\)
	with \(\alpha\neq 1\) may change the class of the same infinite word
	\(u\), because the suffix behaviour on dyadic scales differs.
	
	For example, let \(\Si=\{a,b\}\) and \(u=(ab)^\omega\). With
	\(N_\varepsilon=\lfloor 1/\varepsilon\rfloor\), the truncation
	\(T_\varepsilon(u)\) ends in \(a\) when \(N_\varepsilon\) is odd and in
	\(b\) when \(N_\varepsilon\) is even. Since \(N_\varepsilon\) alternates
	parity on arbitrarily small scales, the suffix behaviour of
	\((T_\varepsilon(u))\) is tied to this specific scale. Replacing
	\(\lfloor 1/\varepsilon\rfloor\) by \(\lfloor\varepsilon^{-2}\rfloor\)
	changes that parity pattern and may change the class
	\(\iota_\omega(u)\). Hence the map is not canonical.
\end{remark}

\begin{lemma}[Injectivity of $\iota^\infty$]\label{lem:iota_infty_injective}
	The map $\iota^\infty:\Si^\infty \to \widetilde{\Si}^*$ is injective.
\end{lemma}
\begin{proof}
	Injectivity on $\Si^*$ is already known
	(Proposition~\ref{prop:iota_injective}).  For two distinct infinite words
	$u,v\in\Si^\omega$, let $k$ be the first position where they differ.
	Because $\lfloor 1/\varepsilon\rfloor\to\infty$, for all sufficiently small
	$\varepsilon$ we have $\lfloor 1/\varepsilon\rfloor\ge k$, and therefore
	$T_\varepsilon(u)$ and $T_\varepsilon(v)$ share the first $k-1$ symbols but
	differ at position $k$.  Hence $\lcp(T_\varepsilon(u),T_\varepsilon(v))=k-1$
	eventually.  A common prefix of bounded length cannot satisfy the equivalence
	condition $\lcp(\cdot,\cdot)\ge m\log_2(1/\varepsilon)$ for every $m$, so
	$(T_\varepsilon(u))\not\sim(T_\varepsilon(v))$.  Thus
	$\iota_\omega(u)\neq\iota_\omega(v)$.  Finally, a finite word and an infinite
	word cannot have equivalent nets: the former has bounded length, while the
	latter has length $\lfloor 1/\varepsilon\rfloor\to\infty$, and the same
	argument shows that equivalence would force an impossibly long common prefix.
	Hence $\iota^\infty$ is injective.
\end{proof}

\begin{theorem}[Proper extension]\label{thm:strict_inclusion}
	Let $\Si = \{\texttt{a},\texttt{b}\}$ and let $[W] = [(w_\varepsilon)]$ be the
	class of the oscillating net~\eqref{eq:osc_net}.  Then, with respect to the
	truncation map $\iota^\infty$ of Definition~\ref{def:iota_infty},
	\begin{enumerate}[label=(\roman*)]
		\item $[W] \notin \iota^\infty(\Si^*)$; \emph{i.e.}\ $[W]$ is not represented
		by a finite word.
		\item $[W] \notin \iota^\infty(\Si^\omega)$; \emph{i.e.}\ no infinite word
		represents $[W]$ under this truncation.
	\end{enumerate}
\end{theorem}

\begin{proof}
	\textbf{(i) Finite words.}
	Suppose $[W] = \iota^\infty(u)$ for some $u\in\Si^*$.  Then the constant net
	$(u)_\varepsilon$ is equivalent to $(w_\varepsilon)$.  By
	Definition~\ref{def:equiv}, for every $m$ we would have
	$\lcp(u,w_\varepsilon) \ge m\log_2(1/\varepsilon)$ for all sufficiently
	small $\varepsilon$.  But $\lcp(u,w_\varepsilon) \le |u|$ for every
	$\varepsilon$, while $m\log_2(1/\varepsilon) \to \infty$ as
	$\varepsilon\to0^+$.  This is impossible; hence
	$[W] \notin \iota^\infty(\Si^*)$.
	
	\textbf{(ii) Infinite words.}
	Assume, for contradiction, that there exists an infinite word
	$u = u_1u_2\cdots \in \Si^\omega$ such that $\iota_\omega(u) = [W]$, i.e.\
	$[(T_\varepsilon(u))] = [(w_\varepsilon)]$.
	
	Because $(T_\varepsilon(u)) \sim (w_\varepsilon)$,
	Definition~\ref{def:equiv} (with $m=2$) yields $\varepsilon_0>0$ such that
	\[
	\lcp(T_\varepsilon(u), w_\varepsilon) \ge 2\log_2(1/\varepsilon)
	\qquad \text{for all } 0<\varepsilon < \varepsilon_0 .
	\]
	In particular, the first letter of $T_\varepsilon(u)$ coincides with the
	first letter of $w_\varepsilon$ whenever $2\log_2(1/\varepsilon) \ge 1$, i.e.\
	for $\varepsilon \le 2^{-1/2}$.
	
	By the definition of truncation, $(T_\varepsilon(u))_1 = u_1$ for every
	$\varepsilon < 1$.  Hence for all sufficiently small $\varepsilon$, the first
	letter of $w_\varepsilon$ must equal $u_1$.
	
	Now consider the sequences $\varepsilon_j = 2^{-2j}$ and
	$\varepsilon'_j = 2^{-(2j+1)}$ ($j \ge 1$).  For sufficiently large $j$, both
	are smaller than $\min(\varepsilon_0, 2^{-1/2})$. By the definition of the
	oscillating net $w_\varepsilon$ in~\eqref{eq:osc_net},
	\[
	w_{\varepsilon_j} \text{ begins with } \texttt{a} \quad (\text{even }
	\lfloor\log_2(1/\varepsilon)\rfloor),\qquad
	w_{\varepsilon'_j} \text{ begins with } \texttt{b} \quad (\text{odd }
	\lfloor\log_2(1/\varepsilon)\rfloor).
	\]
	Thus we obtain $u_1 = \texttt{a}$ (from the first sequence) and
	$u_1 = \texttt{b}$ (from the second), an impossibility.  This contradiction
	shows that no infinite word $u$ represents $[W]$.  Hence
	$[W] \notin \iota_\omega(\Si^\omega)$.
\end{proof}

\begin{corollary}[Strict inclusion]\label{cor:strict_inclusion}
	For the fixed truncation map $\iota^\infty$ of
	Definition~\ref{def:iota_infty},
	$\iota^\infty(\Si^\infty) \subsetneq \widetilde{\Si}^*$.  In particular,
	this truncation injection $\Si^\infty \hookrightarrow \widetilde{\Si}^*$
	is not surjective.
\end{corollary}

The previous results show that $\widetilde{\Si}^*$ contains equivalence
classes, such as $[W]$, that are not represented by any finite or infinite
classical word.  Such classes are genuine \emph{asymptotic ghosts}: they arise
from moderately growing nets whose oscillation prevents convergence in the
Cantor metric, yet they become well-defined objects in the quotient by
$\sim$.  The same class $[W]$ admits a well-defined evaluation of
prefix-dependent moulds (Theorem~\ref{thm:well_def_open}) and, after
logarithmic Ces\`aro renormalization, yields the finite value computed in
Theorem~\ref{thm:renorm_value}.

Thus, for the fixed truncation map \(\iota^\infty\) of
Definition~\ref{def:iota_infty}, \(\widetilde{\Si}^*\) provides an
algebraic framework in which a non-convergent net acquires a
well-defined equivalence class; concatenation of such a class with any
other generalized word is well defined and again gives an element of
the monoid; and prefix-dependent functionals descend to the quotient
and can be evaluated in a well-defined manner after a suitable
renormalization.

None of these features is available in the classical Cantor space
\(\Si^\omega\) or in the union \(\Si^\infty\) as usually construed.
By Corollary~\ref{cor:strict_inclusion}, the chosen truncation map
\(\Si^\infty\hookrightarrow\widetilde{\Si}^*\) is not surjective, so
the generalized monoid contains elements that do not arise from finite
or right-infinite words under that particular embedding.
This provides a monoid construction that addresses the obstruction
described by Perrin and Pin~\cite{PerrinPin2004}: it is a genuine
monoid that contains \(\Si^*\) and admits words of controlled infinite
length, while retaining the ability to evaluate important classes of
functionals.

\section{Further perspectives: scales beyond the logarithmic window}
\label{sec:scales}

The present section briefly outlines how the logarithmic window can be
replaced by other admissible asymptotic scales, leading to a hierarchy
of quotient monoids. A systematic study of these variants is left for
future work; here we only record the basic definitions and monotonicity
properties needed to situate the logarithmic construction.

\begin{definition}[Admissible scale]
	An \emph{admissible scale} is an unbounded nondecreasing function
	\[
	\varphi:(0,1]\longrightarrow\mathbb R_{>0}.
	\]
	Given such a scale, two moderate nets \((u_\varepsilon)\) and
	\((v_\varepsilon)\) are \(\varphi\)-equivalent, written
	\((u_\varepsilon)\sim_\varphi(v_\varepsilon)\), if for every fixed
	\(m\in\mathbb N\) there exists \(\varepsilon_0>0\) such that
	\[
	\min\{\operatorname{lcp}(u_\varepsilon,v_\varepsilon),
	\operatorname{lcs}(u_\varepsilon,v_\varepsilon)\}
	\ge
	m\,\varphi(\varepsilon)
	\qquad(\varepsilon<\varepsilon_0).
	\]
\end{definition}

The logarithmic equivalence of Definition~\ref{def:equiv} is the
special case \(\varphi(\varepsilon)=\log_2(1/\varepsilon)\).

\begin{proposition}[Monotonicity of scales]\label{prop:scale_monotonicity}
	Let \(\varphi,\psi\) be admissible scales.
	\begin{enumerate}[label=(\roman*)]
		\item If \(\varphi=O(\psi)\) and \(\psi=O(\varphi)\), then
		\(\sim_\varphi=\sim_\psi\).
		\item If \(\varphi=o(\psi)\), then \(\sim_\psi\) is finer than
		\(\sim_\varphi\); that is, every \(\psi\)-equivalence class is
		contained in a \(\varphi\)-equivalence class.
	\end{enumerate}
\end{proposition}

\begin{proof}
	Both assertions follow directly from the definition, because the
	eventual inequality \(m\psi(\varepsilon)\ge m'\varphi(\varepsilon)\)
	allows one to pass from the finer scale to the coarser one.
\end{proof}

\begin{example}[Natural and subsequential scales]\label{ex:natural_scales}
	If one indexes words by \(\varepsilon=1/n\) with \(n\in\mathbb N\),
	then
	\[
	\varphi(1/n)=\log_2 n.
	\]
	Restricting \(n\) to a sparse subsequence changes the effective scale
	in the index \(n\):
	\begin{enumerate}[label=(\alph*)]
		\item For \(n=2^k\), one has
		\(\log_2(2^k)=k\), so the logarithmic scale in \(\varepsilon\)
		becomes linear in \(k\).
		\item For \(n=p_k\), the \(k\)-th prime, one has
		\(p_k\sim k\log k\), hence
		\(\log_2 p_k\sim \log_2 k+\log_2\log k\).
		\item For \(n=F_k\), the \(k\)-th Fibonacci number, one has
		\(F_k\sim \phi^k\) with \(\phi=(1+\sqrt5)/2\), hence
		\(\log_2 F_k\sim k\log_2\phi\).
	\end{enumerate}
	Thus the same logarithmic scale in \(\varepsilon\) translates into
	different effective scales in the index \(k\) depending on the
	underlying subsequence.
\end{example}

\begin{example}[Polynomial and exponential scales]\label{ex:poly_exp_scales}
	For \(\alpha>0\), the polynomial scale
$
	\varphi_\alpha(\varepsilon)=\varepsilon^{-\alpha}
	$
	is admissible. Since
$
	\log_2(1/\varepsilon)=o(\varepsilon^{-\alpha}),
	$
	Proposition~\ref{prop:scale_monotonicity} shows that
	\(\sim_{\varphi_\alpha}\) is finer than the logarithmic equivalence.
	Consequently, the polynomial scale is finer than the logarithmic
scale: it may separate nets that are logarithmically equivalent, but
it cannot identify nets that are logarithmically inequivalent.
	
	Similarly, the exponential scale
$
	\varphi_{\exp}(\varepsilon)=2^{\varepsilon^{-1}}
	$
	is admissible and is finer than every polynomial scale.
\end{example}

\begin{remark}[Idempotents and cancellation change with the scale]
	The algebraic properties of the quotient depend strongly on the
	chosen scale. For instance, a unary net \(a^{\ell(\varepsilon)}\)
	represents an idempotent modulo \(\sim_\varphi\) if and only if
	\[
	\frac{\ell(\varepsilon)}{\varphi(\varepsilon)}
	\longrightarrow\infty
	\qquad(\varepsilon\to0^+).
	\]
		Under the logarithmic scale, \(a^{\lfloor\log_2(1/\varepsilon)\rfloor}\)
	is not idempotent. Under a polynomial scale \(\varepsilon^{-\alpha}\),
	it remains non-idempotent, since
	\[
	\frac{\lfloor\log_2(1/\varepsilon)\rfloor}{\varepsilon^{-\alpha}}
	\sim \varepsilon^\alpha \log_2(1/\varepsilon)\to 0.
	\]
	More generally, a unary net can become idempotent only if its length
	dominates the scale; passing to a finer scale therefore makes
	idempotency harder, not easier.
\end{remark}

The general study of the monoids
\(\widetilde{\Si}^*_\varphi:=\cE_M(\Si^*)/\sim_\varphi\) for different
admissible scales, their embeddings into one another, and the behavior
of arithmetic subsequences such as primes or Fibonacci numbers, is left
for future work. The present paper has focused on the logarithmic scale
because it is the natural one for the logarithmic-prefix moulds studied
in Sections~\ref{sec:alg_prop} and~\ref{sec:strict}.

We emphasize that the logarithmic scale was chosen because it is the
natural one for the logarithmic-prefix functionals studied in
Sections~\ref{sec:alg_prop} and~\ref{sec:strict}. The polynomial and
exponential scales are mentioned only to illustrate the flexibility of
the construction; they are not developed in this paper.

\section{Conclusion}
\label{sec:conclusion}

We have constructed a generalized monoid of words \(\widetilde{\Si}^*\)
that extends the free monoid \(\Si^*\) with elements of controlled infinite
length. The construction rests on three ingredients: the bidirectional
prefix--suffix metric \(d_{\mathrm{PS}}\), a polynomial growth condition on
nets of finite words, and an asymptotic equivalence relation identifying
nets whose common prefix and common suffix both grow faster than any
constant multiple of \(\log(1/\varepsilon)\). The resulting quotient is a
monoid that contains \(\Si^*\) faithfully, carries a natural preorder and a
reversal involution, and supports the well-defined evaluation of all
logarithmic-prefix functionals satisfying the moderation conditions
specified in Section~\ref{sec:alg_prop}.

As explained in Section~\ref{sec:scales}, the logarithmic scale used in
the main body is one instance of a much larger family of admissible
asymptotic scales. Replacing it by polynomial, exponential, or
subsequence-adapted scales produces a hierarchy of related quotient
monoids, whose algebraic properties vary with the chosen scale.

	The monoid gives a constructive proposal addressing the obstruction
highlighted by Perrin and Pin~\cite{PerrinPin2004}: to obtain an algebraic
structure that extends the free monoid to infinite words while preserving
everywhere-defined concatenation. It is not claimed to be a canonical
extension of the partial concatenation on \(\Si^\infty\); instead, it
provides a well-defined quotient in which certain truncations of infinite
words can be embedded set-theoretically. The oscillating net constructed in
Section~\ref{sec:strict} shows that \(\widetilde{\Si}^*\) strictly enlarges
the classical frameworks \(\Si^\omega\) and \(\Si^\infty\). Nets with no
limit in the Cantor space and no concatenation in \(\Si^\infty\)
nevertheless acquire well-defined classes, and prefix-dependent moulds can
be assigned renormalized finite values relative to the chosen logarithmic
Ces\`aro renormalization.

From the semigroup-theoretic point of view, several structural questions
remain open. A complete classification of Green's relations
\(\mathcal{L},\mathcal{R},\mathcal{J},\mathcal{H},\mathcal{D}\)
and of the principal ideal lattice would be desirable; the present paper
establishes only partial results for the identity, finite words, and the
unary submonoid. Natural submonoids arising from restricted classes of
representatives---bounded, logarithmic, polynomial, or eventually periodic
nets---also deserve systematic study. It remains unclear whether
\(\widetilde{\Si}^*\) admits a categorical universal property, and a
comparison with prefix-only and suffix-only quotients would clarify the
role of the bidirectional metric. Finally, the relationship with the full
mould calculus of \'Ecalle~\cite{Ecalle1981a,Ecalle1981b,Ecalle1985},
the possible development of a generalized automata theory, and the
hierarchy of quotients associated with general admissible scales
(Section~\ref{sec:scales}) are natural directions for future work.
 A complete solution of these problems, especially the classification of
 Green's relations and of the principal ideal lattice, is the subject of
 ongoing work. The present paper should be regarded as a first
 structural study of \(\widetilde{\Si}^*\), establishing the basic
 monoid-theoretic framework and leaving the finer semigroup-theoretic
 classification to future investigations.

\backmatter

\section*{Declarations}

\subsection*{Funding}
The authors declare that no funds, grants, or other support were received during the preparation of this manuscript.

\subsection*{Conflict of interest}
The authors have no relevant financial or non-financial interests to disclose.

\subsection*{Author contributions}
All authors contributed to the study conception, design, and manuscript preparation. All authors read and approved the final manuscript.

\subsection*{Data availability}
No datasets were generated or analysed during the current study.


\end{document}